%% file: Stokes-Brinkman-aee.tex
\documentclass[12pt]{elsarticle}
\usepackage{amsfonts}
\usepackage{amsmath}
\usepackage{amssymb}
\usepackage{color}
\usepackage{enumerate}
\usepackage{fullpage}
\usepackage{graphicx}
\usepackage[margin=0.75in]{geometry}
\usepackage{mathtools}
\usepackage{multirow}
\usepackage{stmaryrd}

\usepackage[pdftex]{thumbpdf}
\usepackage[pdftex,
pdfstartview=FitH,bookmarks=true,bookmarksnumbered=true,bookmarksopen=true,hypertexnames=false,breaklinks=true,
colorlinks=true,  linkcolor=blue,anchorcolor=blue,
citecolor=blue,filecolor=blue,  menucolor=blue]{hyperref}%
\providecommand{\U}[1]{\protect\rule{.1in}{.1in}}
\usepackage{tikz}
\usetikzlibrary{patterns}

\newtheorem{theorem}{Theorem}

\newtheorem{algorithm}{Algorithm}

\newenvironment{proof}[1][Proof]{\noindent\textbf{#1.} }{\ \rule{0.5em}{0.5em}}
\renewenvironment{frontmatter}{}{}
\newcommand{\R}{\mathbb{R}}

\begin{document}

\begin{frontmatter}
\title{A posteriori error estimates and adaptive mesh refinement \\
  for the Stokes-Brinkman problem\tnoteref{t1}}
\tnotetext[t1]{Supported by the U.S. National Science Foundation under grant DMS1521563.}

\author[umbc]{Kevin Williamson\corref{cor1}}
\ead{kwillia1@umbc.edu}

\author[ctu]{Pavel Burda}
\ead{Pavel.Burda@fs.cvut.cz}

\author[umbc]{Bed\v{r}ich Soused\'{\i}k}
\ead{sousedik@umbc.edu}

\cortext[cor1]{Corresponding author; tel.: +1 410-455-3298; fax: +1 410-455-1066.}

\address[umbc]{Department of Mathematics and Statistics, \\
  University of Maryland, Baltimore County,\\
  1000 Hilltop Circle, Baltimore, MD 21250, USA.}
\address[ctu]{Faculty of Mechanical Engineering, \\
Czech Technical University in Prague,  \\  Technick\'{a} 4, 166 07 Prague 6 - Dejvice, Czech Republic.}

\begin{abstract}
  The Stokes-Brinkman equations model flow in heterogeneous porous media by
  combining the Stokes and Darcy models of flow into a single system of equations. With
  suitable parameters, the equations can model either flow without detailed
  knowledge of the interface between the two regions. Thus, the Stokes-Brinkman
  equations provide an alternative to coupled Darcy-Stokes models.
  After a brief review of the Stokes-Brinkman problem and its discretization
  using Taylor-Hood finite elements, we present a residual-based a posteriori
  error estimate and use it to drive an adaptive mesh refinement process.
  We compare several strategies for the mesh refinement, and demonstrate its
  effectiveness by numerical experiments in both 2D and 3D.
\end{abstract}


\maketitle
\end{frontmatter}

\input{aee_paper-introduction}
\input{aee_paper-stokes-brinkman}
\input{aee_paper-aee}
\input{aee_paper-refinement}
\input{aee_paper-numerical}
\input{aee_paper-conclusion}

\medskip
\begin{center} {\textbf{Dedication}} \end{center}
{\textit{This article is dedicated to the memory of Prof. Ivo Marek, unforgettable mentor and colleague.}}

\bibliographystyle{plain}
\bibliography{Stokes-Brinkman-aee}

\end{document}

%% file: aee_paper-introduction.tex
\section{Introduction}

The simulation of flow in porous media has numerous applications, to include reservoir simulation, nuclear waste disposal, and carbon dioxide sequestration. Such simulation is challenging for a variety of reasons. First, the domains tend to be fairly irregular, which complicates the model geometry. Second, the geologic formations consist of many varying materials, with different geologic properties. Third, there are often fractures and vugs within the domain that alter the effective permeabilities.
The standard approach to modeling these types of problems is to couple Darcy and Stokes and enforce the Beavers-Joseph-Saffman conditions along the interface~\cite{Arbogast2007,Beavers-1967-FPM,URQUIZA2008525}. The free-flow regions (fractures, vugs) are modeled using Stokes flow, whereas the porous region is modeled using Darcy's law~\cite{Darcy-1856-FPM,Whitaker-1986-FPM}. 
However, the two types of domains are not well-separated in reservoirs, and it may be difficult to determine the appropriate conditions to enforce along the interface. 

We model flow in porous media using the Stokes-Brinkman equations~\cite{Bear-1972-PM, Brinkman-1948-SBE, Gulbransen-2010-MMF, Laptev-2003-PPM, Popov-2007-MMM}, which combine Stokes and Darcy into a single system of equations. The Stokes-Brinkman equations reduce to Stokes or Darcy flow depending upon the coefficients and were suggested as a replacement for coupled Darcy and Stokes in~\cite{Popov-2007-MMM}. By careful selection of coefficients, the equations allow modeling of free-flow and porous domains together, thereby resolving issues along the interface.

In this paper, we present a residual-based a posteriori error estimate for the Stokes-Brinkman problem discretized using
Taylor-Hood finite elements, and we use it to drive an adaptive mesh refinement process. 
There are various strategies for a posteriori error estimation presented in the literature, see 
for example~\cite{Ainsworth-2000-AEE,Eriksson-1995-IAM,Verfurth-2013-AEE}.
The closest to the presented work are the estimates developed for the Stokes and Navier-Stokes equations
by Burda et al.~\cite{Burda-2000-AEE,Burda-2001-AEE,Burda-2003-AEE}, 
and in particular for the Stokes-Brinkman problem derived in 2D by Burda and Hasal~\cite{Burda-2015-AEE}. 
Here, we first extend the estimate from~\cite{Burda-2015-AEE} to~3D. 
Then, we use the estimate to drive an adaptive mesh refinement process. 
Finally, we study several mesh refinement strategies 
and present numerical experiments in both~2D and~3D. 

The paper is organized as follows. 
In Section~\ref{sec:sb}, we introduce the Stokes-Brinkman problem and its finite element discretization using Taylor-Hood elements. The error estimate is presented in Section~\ref{sec:aee}, and several mesh refinement strategies to be studied are presented in Section~\ref{sec:refinement}. In Section~\ref{sec:numerical}, we present the results of numerical experiments. Finally, in Section~\ref{sec:conclusion}, we summarize and conclude our work.

%% file: aee_paper-stokes-brinkman.tex
\section{Stokes-Brinkman problem} \label{sec:sb}

Let $\Omega \subset \R^d$, $d=2,3$ be a connected, open domain with Lipschitz boundary $\partial \Omega$. Let $\Omega_f \subset \Omega$ be a free-flow region within $\Omega$ and $\Omega_p \subset \Omega$ be a porous medium within $\Omega$ such that $\Omega_f \cap \Omega_p = \emptyset$ and $\overline{\Omega} = \overline{\Omega_f} \cup \overline{\Omega_p}$. Within the free-flow region, $\Omega_f$, the flow is governed by the Stokes equations
\begin{align}
  -\mu \Delta \vec{u} + \nabla p &= \vec{f} \label{eq:stokes-momentum} \\
  \nabla \cdot \vec{u} &= g, \label{eq:stokes-mass}
\end{align}
where $\mu$ is the viscosity of the fluid, $\vec{u} \colon \R^d \rightarrow \R^d$ is the velocity, $p \colon \R^d \rightarrow \R$ is the pressure, $\vec{f} \colon \R^d \rightarrow \R^d$ denotes external forces, and $g \colon \R^d \rightarrow \R$ denotes sources and sinks. Equation~(\ref{eq:stokes-momentum}) is derived from the conservation of momentum, and equation~(\ref{eq:stokes-mass}) is derived from the conservation of mass and denotes the incompressibility of the fluid. Within the porous media region, $\Omega_p$, the flow is governed by Darcy's law
\begin{align}
  \vec{u} &= -\frac{K}{\mu} (\nabla p - \vec{f}) \label{eq:darcy1} \\
  \nabla \cdot \vec{u} &= g, \label{eq:darcy2}
\end{align}
where $K$ is a symmetric positive definite permeability tensor. It is important to note that the velocity and pressure terms, $\vec{u}$ and $p$, are different than the same terms in the Stokes equations. In the Stokes equations, they denote the actual velocity and pressure of the fluid, whereas, in Darcy's law, they denote the averages over some representative element volume.

In the Stokes-Brinkman equations, the Stokes~(\ref{eq:stokes-momentum})--(\ref{eq:stokes-mass}) and Darcy~(\ref{eq:darcy1})--(\ref{eq:darcy2}) flows are combined into a single system of equations
\begin{align}
  -\mu^* \Delta \vec{u} + \mu K^{-1}\vec{u} + \nabla p &= \vec{f} \label{eq:sb-momentum} \\
  \nabla \cdot \vec{u} &= g, \label{eq:sb-mass}
\end{align}
where $\mu^*$ denotes the effective viscosity of the fluid. Both Stokes and Darcy flows are limiting cases of the Stokes-Brinkman equations using suitable choices of $\mu^*$ and $K$. If $\mu^* = 0$, Equation~(\ref{eq:sb-momentum}) is simply Darcy's law; whereas, if $K \gg 0$, it reduces to the Stokes equations. Within $\Omega_p$, we choose $K$ to be the Darcy permeability; within $\Omega_f$, we choose $K^{-1} = 0$. The choice of $\mu^*$ is crucial for resolving conditions between the free-flow and porous interface. If detailed knowledge is available, $\mu^*$ may be chosen to mimic the Beavers-Joseph-Saffman conditions along the interface. However, absent such detailed knowledge, we may select $\mu^* = \mu$ throughout the entire domain. This has the consequence of only a slight perturbation of Darcy's law in the porous domain~\cite{Laptev-2003-PPM, Popov-2007-MMM}.

The Stokes-Brinkman equations are accompanied by Dirichlet and Neumann conditions of the form
\begin{align}
  \vec{u} &= \vec{u}_D \quad \textrm{ on } \partial \Omega_D \label{eq:dirichlet} \\
  \frac{\partial \vec{u}}{\partial n} - p\vec{n}&= \vec{u}_N \quad \textrm{ on } \partial \Omega_N, \label{eq:neumann}
\end{align}
where $\partial \Omega_D$ and $\partial \Omega_N$ denote the Dirichlet and Neumann parts of the boundary, respectively, $\vec{n}$ is the unit outward normal, and $\frac{\partial \vec{u}}{\partial n}$ is the directional derivative of the velocity in the normal direction.

We seek a weak solution to equations~(\ref{eq:sb-momentum})--(\ref{eq:neumann}). Let us define the spaces
\begin{align}
  H^1_E(\Omega) &= \{ \vec{u} \in H^1(\Omega)^d \mid \vec{u} = \vec{u}_D \textrm{ on } \partial \Omega_D\} \nonumber \\
  H^1_{E_0}(\Omega) &= \{ \vec{u} \in H^1(\Omega)^d \mid \vec{u} = 0 \textrm{ on } \partial \Omega_D\}, \nonumber
\end{align}
and define the bilinear forms
\begin{align}
  a(\vec{u}, \vec{v}) &= \int_{\Omega} \left( \mu^* \nabla \vec{u} \colon \nabla \vec{v} + \mu \vec{u}^T K^{-1} \vec{v} \right) dx \label{eq:bilinear-a} \\
  b(\vec{u}, q) &= -\int_{\Omega} q \, \nabla \cdot \vec{u} dx , \label{eq:bilinear-b}
\end{align}
where $\nabla \vec{u} \colon \nabla \vec{v} = \sum_{i=1}^d \nabla \vec{u_i} \cdot \nabla \vec{v}_i$. In the weak formulation of Stokes-Brinkman, we wish to find $\vec{u} \in H^1_E(\Omega)$ and $p \in L_2(\Omega)$ such that
\begin{align}
  a(\vec{u}, \vec{v}) + b(\vec{v}, p) &= \int_{\Omega} \vec{f} \cdot \vec{v} dx + \int_{\partial \Omega_N} \vec{u}_N \cdot \vec{v} ds, \quad \forall \vec{v} \in H^1_{E_0}(\Omega) \label{eq:sb-weak-momentum} \\
  b(\vec{u}, q) &= -\int_{\Omega} gq dx, \quad \forall q \in L_2(\Omega). \label{eq:sb-weak-mass}
\end{align}
We need to set proper boundary conditions. In order to guarantee a unique velocity solution, the Dirichlet part of the boundary, $\partial \Omega_D$,  must have nonzero measure. Similarly, in order to guarantee a unique pressure solution, the Neumann part of the boundary, $\partial \Omega_N$, must have nonzero measure. In the case where $\partial \Omega_N$ has measure zero, the pressure solution is unique up to a constant, so we may impose the additional constraint $\int_{\Omega} p dx = 0$ in order to obtain a unique pressure. Furthermore, in such cases, we must impose a compatibility condition on the boundary data
\begin{displaymath}
  \int_{\partial \Omega_{+}} \vec{u}_D \cdot \vec{n} ds - \int_{\partial \Omega_{-}} \vec{u}_D \cdot \vec{n} ds = \int_{\partial \Omega} g ds ,
\end{displaymath}
where $\partial \Omega_{+} = \{ x \in \partial \Omega \mid \vec{u}_D \cdot \vec{n} > 0\}$ and $\partial \Omega_{-} = \{ x \in \partial \Omega \mid \vec{u}_D \cdot \vec{n} < 0\}$ are the outflow and inflow boundaries, respectively. In our experiments, we use Dirichlet conditions on the inflow and so-called \emph{do-nothing} conditions on the outflow. For more details, see, for example,~\cite[Chapter~3]{Elman-2014-FEF}.

We use the mixed finite element method to discretize~(\ref{eq:sb-weak-momentum})--(\ref{eq:sb-weak-mass}). Let $V_h \subset H_{E_0}^1(\Omega)$ denote the discretized space of velocities with basis $\{\psi_1,\psi_2,\ldots,\psi_n\}$ and $P_h \subset L^2(\Omega)$ the discretized space of pressures with basis $\{\phi_1,\phi_2,\ldots,\phi_m\}$. To incorporate the Dirichlet boundary conditions, we extend $V_h$ by defining additional basis functions $\psi_{n+1},\ldots,\psi_{n + n_\partial}$ and coefficients $u_j$, $j=n+1,\ldots,n+n_\partial$, such that $\sum_{j=n+1}^{n+n_\partial} u_j \psi_j$ interpolates the boundary data, $\vec{u}_D$. We then seek a finite element solution of the form $\vec{u}_h = \sum_{j=1}^n u_j \psi_j + \sum_{j=n+1}^{n+n_\partial} u_j \psi_j$ and $p_h = \sum_{j=1}^m p_j \phi_j$. The discrete weak formulation of Stokes-Brinkman is to find $\mathbf{u} = (u_1,\ldots,u_n)^T \in \R^n$ and $\mathbf{p} = (p_1,\ldots,p_m)^T \in \R^m$ such that
\begin{align}
  a\left(\sum_{j=1}^n u_j \psi_j + \sum_{j=n+1}^{n+n_\partial} u_j \psi_j, \vec{v}_h\right) + b\left(\vec{v}_h, \sum_{j=1}^m p_j \phi_j\right) &= \int_{\Omega} \vec{f} \cdot \vec{v}_h dx + \int_{\partial \Omega_N} \vec{u}_N \cdot \vec{v}_h ds, \quad \forall \vec{v}_h \in V_h \label{eq:discrete-weak1} \\
  b\left(\sum_{j=1}^n u_j \psi_j + \sum_{j=n+1}^{n+n_\partial} u_j \psi_j, q_h\right) &= -\int_{\Omega} gq_h dx, \quad \forall q_h \in P_h. \label{eq:discrete-weak2}
\end{align}
As the above must hold for all $(\vec{v}_h, q_h) \in V_h \times P_h$, we may choose $\vec{v}_h = \psi_i$ and $q_h = \phi_i$, so that we obtain the saddle point system
\begin{equation}
  \left[
    \begin{array}{cc}
      \mathbf{A} & \mathbf{B}^T \\
      \mathbf{B} & 0
    \end{array}
    \right]
  \left[
    \begin{array}{c}
      \mathbf{u} \\
      \mathbf{p}
    \end{array}
    \right] =
  \left[
    \begin{array}{c}
      \mathbf{f} \\
      \mathbf{g}
    \end{array}
    \right]. \nonumber
\end{equation}
The matrices $\mathbf{A}$ and $\mathbf{B}$ are discrete versions of the bilinear forms~(\ref{eq:bilinear-a}) and~(\ref{eq:bilinear-b}), respectively. The vectors $\mathbf{u}$ and $\mathbf{p}$ denote the coefficients of the discrete velocity and pressure, respectively, in the chosen finite element basis, and $\mathbf{f}$ and $\mathbf{g}$ are discretizations of the external forces and sources and sinks, respectively. We use Taylor-Hood P2/P1 finite elements to guarantee inf-sup stability, see, for example, the monographs~\cite{Brezzi-1991-MHF,Elman-2014-FEF}. This results in piecewise (bi)-quadratic approximation of the velocity and piecewise linear approximation of the pressure.

%% file: aee_paper-aee.tex
\section{A posteriori error estimate} \label{sec:aee}

Let $\vec{u}_h$ and $p_h$ denote the finite element solution using Taylor-Hood P2/P1 finite elements. We define the residuals $\mathbf{R_1}$ and $R_2$ of the discrete counterparts of~(\ref{eq:sb-momentum})--(\ref{eq:sb-mass}) as
\begin{align}
  \mathbf{R_1}(\vec{u}_h, p_h) &= \vec{f} + \mu^* \Delta \vec{u}_h - \mu K^{-1}\vec{u}_h - \nabla p_h \label{eq:residual-R1} \\
  R_2(\vec{u}_h, p_h) &= g - \nabla \cdot \vec{u}_h. \label{eq:residual-R2}
\end{align}
Because $u_h$ is piecewise (bi)-quadratic, $\Delta \vec{u}_h$ and $\nabla \cdot \vec{u}_h$ may be computed on element interiors, even though $u_h$ may have discontinuous first derivatives along element interfaces. Furthermore, since $p_h$ is piecewise linear, $\nabla p_h$ may be computed on element interiors, despite $p_h$ possibly having discontinuous first derivatives along element interfaces. Thus, within element interiors, the computation of the residuals~(\ref{eq:residual-R1})--(\ref{eq:residual-R2}) is well-defined. Due to the discontinuous derivatives at element interfaces, we define the flux jumps as follows. Given an edge/face $E$, let $T$ and $T\prime$ denote the elements sharing $E$ and $\vec{n}$ denote the unit outward normal of $E$ on $T$. Then the flux jump is defined as
\begin{equation*}
  \left \llbracket \mu \frac{\partial \vec{u}_h}{\partial n} - p_h \vec{n} \right \rrbracket_E = \left(\mu \frac{\partial \vec{u}_h}{\partial n} - p_h \vec{n}\right)_+ - \left(\mu \frac{\partial \vec{u}_h}{\partial n} - p_h \vec{n}\right)_-
\end{equation*}
where $(\cdot)_+$ uses values of $(\vec{u}_h,p_h)$ within $T$ and $(\cdot)_-$ uses values of $(\vec{u}_h,p_h)$ within $T\prime$. Note that, if using a continuous pressure approximation as with Taylor-Hood P2/P1 elements, the pressure terms in the flux jumps vanish. From the flux jumps, we define the equilibrated edge residuals
\begin{equation*}
  \mathbf{R_E}(\vec{u}_h,p_h) = \left\{
  \begin{array}{ll}
    \frac{1}{2} \left \llbracket \mu \frac{\partial \vec{u}_h}{\partial n} - p_h \vec{n} \right \rrbracket_E, & E \in \mathcal{E}_{h,\Omega} \\
    \vec{u}_N - \left(\mu \frac{\partial \vec{u}_h}{\partial n} - p_h \vec{n} \right), & E \in \mathcal{E}_{h,N} \\
    0, & E \in \mathcal{E}_{h,D}
  \end{array} \right.
\end{equation*}
where $\mathcal{E}_{h,\Omega}, \mathcal{E}_{h,N}, \mathcal{E}_{h,D}$ denote the interior, Neumann, and Dirichlet edges (or faces), respectively.

We now revisit the main theorem of~\cite{Burda-2015-AEE}. The theorem was proven there for the special case $d=2$ and $g = 0$, and we state it here also for the case $d=3$ and general $g$. The extension is relatively straightforward, but we include the full proof for completeness.

\begin{theorem} \label{thm:aee-estimate}
  Let $\Omega$ be a polygon in $\R^2$ or a polyhedron in $\R^3$. Let $\mathcal{T}_h$ be a family of regular triangulations of $\Omega$. Let $(\vec{u}_h, p_h)$ be the Taylor-Hood approximation of the solution $(\vec{u},p)$ of the Stokes-Brinkman problem. Then the error $(\vec{e}_u,e_p) = (\vec{u} - \vec{u}_h,p-p_h)$ satisfies the following a posteriori error estimate

  \begin{equation}
    \|\vec{e}_u\|_1 + \|e_p\|_0 \leq 2C_PC_IC_R \sum_{T \in \mathcal{T}_h} \left(h_T \|\mathbf{R_1}(\vec{u}_h,p_h)\|_{0,T} + \|R_2(\vec{u}_h, p_h)\|_{0,T} + h_T^{1/2} \sum_{E \in \mathcal{E}(T)} \| \mathbf{R_E}(\vec{u}_h,p_h) \|_{0,E}\right) \label{eq:aee-estimate}
  \end{equation}

  \noindent where $C_P, C_I, C_R$ are positive constants, $h_T$ is the diameter of element $T$, and $\mathcal{E}(T)$ denotes the set of edges (in 2D) or faces (in 3D) of element $T$.
\end{theorem}

\begin{proof}
The ideas are based upon Eriksson et al~\cite{Eriksson-1995-IAM}. First, we recall the Poincar\'e-Friedrichs inequality
\begin{equation}
  \|\vec{e}_u\|_1^2 \leq C_P  \|\nabla \vec{e}_u\|_0^2 \label{eq:poincare-friedrichs}
\end{equation}
for a constant $C_P \geq 1$. We next define the dual Stokes-Brinkman problem
\begin{align}
  -\mu^* \Delta \vec{\varphi}_u + \mu K^{-1} \vec{\varphi}_u + \nabla \varphi_p &= -\Delta \vec{e}_u \label{eq:dual-sb1} \\
  \nabla \cdot \vec{\varphi}_u &= -e_p \label{eq:dual-sb2} \\
  \vec{\varphi}_u &= 0 \qquad \textrm{on } \partial \Omega. \label{eq:dual-sb3}
\end{align}
The weak form of~(\ref{eq:dual-sb1})--(\ref{eq:dual-sb3}) is to find $(\vec{\varphi}_u, \varphi_p) \in (H_0^1(\Omega))^d \times L^2(\Omega)$ such that
\begin{align}
  a(\vec{\varphi}_u, \vec{v}) + b(\vec{v}, \varphi_p) &= \int_{\Omega} \nabla \vec{e}_u \colon \nabla \vec{v} dx , \quad \forall \, \vec{v} \in (H^1(\Omega))^d \label{eq:dual-sb-weak1} \\
  b(\vec{\varphi}_u, q) &= \int_{\Omega} e_p q dx , \quad \forall \, q \in L^2(\Omega). \label{eq:dual-sb-weak2}
\end{align}
Since $\vec{e}_u \in (H^1(\Omega))^d$ and $ e_p \in L^2(\Omega)$, we may choose $\vec{v} = \vec{e}_u$ in~(\ref{eq:dual-sb-weak1}) and $q = e_p$ in~(\ref{eq:dual-sb-weak2}), and use~(\ref{eq:poincare-friedrichs}) to obtain
\begin{align}
  \frac{1}{C_P} \|\vec{e}_u\|_1^2 & \leq \|\nabla \vec{e}_u\|_0^2 = \int_{\Omega} \nabla \vec{e}_u \colon \nabla \vec{e}_u dx \nonumber \\
  \, & = a(\vec{\varphi}_u,\vec{e}_u) + b(\vec{e}_u, \varphi_p) \nonumber \\
  \, & = a(\vec{\varphi}_u, \vec{u}) - a(\vec{\varphi}_u, \vec{u}_h) + b(\vec{u}, \varphi_p) - b(\vec{u}_h, \varphi_p) \label{eq:estimate1-velocity} \\
  \|e_p\|_0^2 &= b(\vec{\varphi}_u, e_p) \nonumber \\
  \, & = b(\vec{\varphi}_u, p) - b(\vec{\varphi}_u, p_h). \label{eq:estimate1-pressure}
\end{align}
Since $C_P \geq 1$, we combine~(\ref{eq:estimate1-velocity}) and~(\ref{eq:estimate1-pressure}) as
\begin{align}
  \frac{1}{C_P}\left(\|\vec{e}_u\|_1^2 + \|e_p\|_0^2\right) & \leq a(\vec{\varphi}_u, \vec{u}) - a(\vec{\varphi}_u, \vec{u}_h) + b(\vec{u}, \varphi_p) - b(\vec{u}_h, \varphi_p) + b(\vec{\varphi}_u, p) - b(\vec{\varphi}_u, p_h) \nonumber \\
  \, & = \left[a(\vec{\varphi}_u, \vec{u}) + b(\vec{u}, \varphi_p) + b(\vec{\varphi}_u, p) \right] - \left[a(\vec{\varphi}_u, \vec{u}_h) + b(\vec{u}_h, \varphi_p) + b(\vec{\varphi}_u, p_h) \right]. \nonumber
\end{align}
We may also choose $\vec{v} = \vec{\varphi}_u$ in~(\ref{eq:sb-weak-momentum}) and $q = \varphi_p$ in~(\ref{eq:sb-weak-mass}), and use the definition of the residuals~(\ref{eq:residual-R1})--(\ref{eq:residual-R2}) to yield
\begin{align}
  \frac{1}{C_P}\left(\|\vec{e}_u\|_1^2 + \|e_p\|_0^2\right) & \leq \int_{\Omega} \vec{f} \cdot \vec{\varphi}_u dx  - \int_{\Omega} g \varphi_p dx - \left[a(\vec{\varphi}_u, \vec{u}_h) + b(\vec{u}_h, \varphi_p) + b(\vec{\varphi}_u, p_h) \right] \nonumber \\
  \, & = \int_{\Omega} \vec{f} \cdot \vec{\varphi}_u dx - \int_{\Omega} g \varphi_p dx \nonumber \\
  \, & \quad + \sum_{T \in \mathcal{T}_h} \left[
  \begin{array}[c]{c}
    \int_{T} \mu^* \Delta \vec{u}_h \cdot \vec{\varphi}_u dx - \int_{\partial T} \mu^* \frac{\partial \vec{u}_h}{\partial n} \cdot \vec{\varphi}_u ds - \int_{T} \mu \vec{\varphi}_u^T K^{-1} \vec{u}_h dx \\
     - \int_{T} \nabla p_h \cdot \vec{\varphi}_u dx + \int_{\partial T} p_h \vec{\varphi}_u \cdot \vec{n} ds + \int_{T} \varphi_p \nabla \cdot \vec{u}_h dx 
  \end{array} \right] \nonumber \\
\, & = \sum_{T \in \mathcal{T}_h} \int_{T} \left(f + \mu^* \Delta \vec{u}_h - \mu K^{-1}\vec{u}_h - \nabla p_h \right) \cdot \vec{\varphi}_u dx + \sum_{T \in \mathcal{T}_h} \int_{T} \left( \nabla \cdot \vec{u}_h - g \right) \varphi_p dx  \nonumber \\
  \, & \quad - \sum_{T \in \mathcal{T}_h} \int_{\partial T} \mu^* \frac{\partial \vec{u}_h}{\partial n} \cdot \vec{\varphi}_u ds + \sum_{T \in \mathcal{T}_h} \int_{\partial T} p_h \vec{\varphi}_u \cdot \vec{n} ds \nonumber \\
\, & = \sum_{T \in \mathcal{T}_h} \int_{T} \mathbf{R_1}(\vec{u}_h,p_h) \cdot \vec{\varphi}_u dx - \sum_{T \in \mathcal{T}_h} \int_{T} R_2(\vec{u}_h, p_h) \varphi_p dx \nonumber \\
  \, & \quad - \sum_{T \in \mathcal{T}_h} \int_{\partial T} \mu^* \frac{\partial \vec{u}_h}{\partial n} \cdot \vec{\varphi}_u ds + \sum_{T \in \mathcal{T}_h} \int_{\partial T} p_h \vec{\varphi}_u \cdot \vec{n} ds. \label{eq:aee-estimate-part2}
\end{align}
By taking $\vec{v}_h = \pi_h\vec{\varphi}_u$ and $q_h = \pi_h \varphi_p$, the Cl\'ement interpolants~\cite{Ciarlet-2002-FEM}, in~(\ref{eq:discrete-weak1}) and~(\ref{eq:discrete-weak2}), we obtain
\begin{align}
  0 &= \int_{\Omega} \vec{f} \cdot \pi_h\vec{\varphi}_u dx - \int_{\Omega} g \pi_h \varphi_p dx - \left[a(\vec{u}_h,\pi_h\vec{\varphi}_u) + b(\vec{u}_h, \pi_h\varphi_p) + b(\pi_h\vec{\varphi}_u, p_h) \right] \nonumber \\
  \, &= \sum_{T \in \mathcal{T}_h} \int_{T} \mathbf{R_1}(\vec{u}_h,p_h) \cdot \pi_h \vec{\varphi}_u dx - \sum_{T \in \mathcal{T}_h} \int_{T} R_2(\vec{u}_h, p_h) \pi_h \varphi_p dx \nonumber \\
  \, & \quad - \sum_{T \in \mathcal{T}_h} \int_{\partial T} \mu^* \frac{\partial \vec{u}_h}{\partial n} \cdot \pi_h \vec{\varphi}_u ds + \sum_{T \in \mathcal{T}_h} \int_{\partial T} p_h \pi_h \vec{\varphi}_u \cdot \vec{n} ds . \nonumber
\end{align}
Subtracting this from~(\ref{eq:aee-estimate-part2}), we obtain
\begin{align}
  \frac{1}{C_P}\left(\|\vec{e}_u\|_1^2 + \|e_p\|_0^2\right) & \leq \sum_{T \in \mathcal{T}_h} \int_{T} \mathbf{R_1}(\vec{u}_h,p_h) \cdot (\vec{\varphi}_u - \pi_h \vec{\varphi}_u) dx - \sum_{T \in \mathcal{T}_h} \int_{T} R_2(\vec{u}_h, p_h) (\varphi_p - \pi_h \varphi_p) dx \nonumber \\
  \, & \quad - \sum_{T \in \mathcal{T}_h} \int_{\partial T} \mu^* \frac{\partial \vec{u}_h}{\partial n} \cdot (\vec{\varphi}_u - \pi_h \vec{\varphi}_u) ds + \sum_{T \in \mathcal{T}_h} \int_{\partial T} p_h (\vec{\varphi}_u - \pi_h \vec{\varphi}_u) \cdot \vec{n} ds  \nonumber \\
  \, & = \sum_{T \in \mathcal{T}_h} \int_{T} \mathbf{R_1}(\vec{u}_h,p_h) \cdot (\vec{\varphi}_u - \pi_h \vec{\varphi}_u) dx - \sum_{T \in \mathcal{T}_h} \int_{T} R_2(\vec{u}_h, p_h) (\varphi_p - \pi_h \varphi_p) dx \nonumber \\
  \, & \quad - \sum_{T \in \mathcal{T}_h} \sum_{E \in \mathcal{E}(T)} \int_{E} \left( \frac{1}{2} \left \llbracket \mu^* \frac{\partial \vec{u}_h}{\partial n} - p_h \vec{n} \right \rrbracket_E \right) (\vec{\varphi}_u - \pi_h \vec{\varphi}_u) ds . \nonumber
\end{align}
In the above, the sum over $E \in \mathcal{E}(T)$ is taken over edges of the triangle in 2D and faces of the tetrahedron in~3D. Using the Schwarz inequality,
\begin{align}
  \|\vec{e}_u\|_1^2 + \|e_p\|_0^2 & \leq C_P \sum_{T \in \mathcal{T}_h} \bigg( \| \mathbf{R_1}(\vec{u}_h,p_h) \|_{0,T} \|\vec{\varphi}_u - \pi_h \vec{\varphi}_u \|_{0,T} + \| R_2(\vec{u}_h, p_h) \|_{0,T} \| \varphi_p - \pi_h \varphi_p \|_{0,T} \bigg) \nonumber \\
  \, & \quad + C_P \sum_{T \in \mathcal{T}_h} \sum_{E \in \mathcal{E}(T)} \left\| \frac{1}{2} \left \llbracket \mu^* \frac{\partial \vec{u}_h}{\partial n} - p_h \vec{n} \right \rrbracket_E \right\|_{0,E} \| \vec{\varphi}_u - \pi_h \vec{\varphi}_u \|_{0,E}. \nonumber
\end{align}
Using the properties of the interpolants (cf.~\cite{Brezzi-1991-MHF}), there exists a constant $C_I > 0$ such that
\begin{align}
  \|\vec{\varphi}_u - \pi_h \vec{\varphi}_u \|_{0,T} & \leq C_I h_T \|\vec{\varphi}_u\|_{1,T} \nonumber \\
  \| \varphi_p - \pi_h \varphi_p \|_{0,T} & \leq C_I \|\varphi_p\|_{0,T} \nonumber \\
  \| \vec{\varphi}_u - \pi_h \vec{\varphi}_u \|_{0,E} & \leq C_I h_T^{1/2} \|\vec{\varphi}_u\|_{1,T}, \nonumber
\end{align}
where $h_T$ is the diameter of element $T$. Thus, we obtain
\begin{align}
  \|\vec{e}_u\|_1^2 + \|e_p\|_0^2 & \leq C_PC_I \sum_{T \in \mathcal{T}_h} \bigg( h_T \| \mathbf{R_1}(\vec{u}_h,p_h) \|_{0,T} \|\vec{\varphi}_u\|_1 + \| R_2(\vec{u}_h, p_h) \|_{0,T} \| \varphi_p \|_0 \bigg) \nonumber \\
  \, & \quad + C_PC_I \sum_{T \in \mathcal{T}_h} \sum_{E \in \mathcal{E}(T)} h_T^{1/2} \left\| \frac{1}{2} \left \llbracket \mu^* \frac{\partial \vec{u}_h}{\partial n} - p_h \vec{n} \right \rrbracket_E \right\|_{0,E} \| \vec{\varphi}_u \|_1. \nonumber
\end{align}
We then use properties of the dual solution (cf.~\cite{Brezzi-1991-MHF}) to derive another constant $C_R > 0$ such that
\begin{align}
  \|\vec{e}_u\|_1^2 + \|e_p\|_0^2 & \leq C_P C_I C_R \sum_{T \in \mathcal{T}_h} \left[
    \begin{array}[c]{c}
      h_T \| \mathbf{R_1}(\vec{u}_h,p_h) \|_{0,T} + \| R_2(\vec{u}_h, p_h) \|_{0,T} \\
      + \sum_{E \in \mathcal{E}(T)} h_T^{1/2} \left\| \frac{1}{2} \left \llbracket \mu^* \frac{\partial \vec{u}_h}{\partial n} - p_h \vec{n} \right \rrbracket_E \right\|_{0,E}
    \end{array} \right] (\| \Delta \vec{e}_u\|_{-1} + \|e_p\|_0). \nonumber
\end{align}
Finally, using the inequality $\| \Delta \vec{e}_u \|_{-1} \leq \| \vec{e}_u \|_1$,
\begin{align}
  \left(\|\vec{e}_u\|_1 + \|e_p\|_0 \right)^2 & \leq 2 \left(\| \vec{e}_u \|_1^2 + \| e_p \|_0^2 \right) \nonumber \\
  \, & \leq 2 C_P C_I C_R \sum_{T \in \mathcal{T}_h} \left[
    \begin{array}[c]{c}
      h_T \| \mathbf{R_1}(\vec{u}_h,p_h) \|_{0,T} + \| R_2(\vec{u}_h, p_h) \|_{0,T} \\
      + \sum_{E \in \mathcal{E}(T)} h_T^{1/2} \left\| \frac{1}{2} \left \llbracket \mu^* \frac{\partial \vec{u}_h}{\partial n} - p_h \vec{n} \right \rrbracket_E \right\|_{0,E}
    \end{array} \right] (\| \vec{e}_u\|_1 + \|e_p\|_0). \nonumber
\end{align}
The a posteriori error estimate in Theorem~\ref{thm:aee-estimate} is obtained by canceling a factor of $\|\vec{e}_u\|_1 + \|e_p\|_0$.
\end{proof}

The error estimate~(\ref{eq:aee-estimate}) in Theorem~\ref{thm:aee-estimate} is defined over the norm $\|\vec{e}_u\|_1 + \|e_p\|_0$. We would prefer the estimate in the norm $(\|\vec{e}_u\|_1^2 + \|e_p\|_0^2)^{1/2}$. To this effect, we note
\begin{align}
  \|\vec{e}_u\|_1^2 + \|e_p\|_0^2 &= \sum_{T \in \mathcal{T}_h} \left(\|\vec{e}_u\|_{1,T}^2 + \|e_p\|_{0,T}^2\right) \nonumber \\
  \, &\leq \sum_{T \in \mathcal{T}_h} \left(\|\vec{e}_u\|_{1,T} + \|e_p\|_{0,T} \right)^2 \nonumber \\
  \, &\leq \sum_{T \in \mathcal{T}_h} C_1 \left(h_T\|\mathbf{R_1}(\vec{u}_h,p_h)\|_{1,T} + \|R_2(\vec{u}_h,p_h)\|_{0,T} + h_T^{1/2}\sum_{E \in \mathcal{E}(T)} \|\mathbf{R_E}(\vec{u}_h,p_h)\|_{0,E}\right)^2 \nonumber \\
  \, &\leq \sum_{T \in \mathcal{T}_h} C_2 \left(h_T^2\|\mathbf{R_1}(\vec{u}_h,p_h)\|_{1,T}^2 + \|R_2(\vec{u}_h,p_h)\|_{0,T}^2 + h_T\sum_{E \in \mathcal{E}(T)} \|\mathbf{R_E}(\vec{u}_h,p_h)\|_{0,E}^2\right) \nonumber
\end{align}
for some constants $C_1,C_2 > 0$, where $C_1$ depends upon $C_P,C_I,C_R$. By defining our elementwise error indicator $\eta_T$ as
\begin{align}
  \eta_T^2 &= h_T^2\|\mathbf{R_1}(\vec{u}_h,p_h)\|_{1,T}^2 + \|R_2(\vec{u}_h,p_h)\|_{0,T}^2 + h_T \sum_{E \in \mathcal{E}(T)} \|\mathbf{R_E}(\vec{u}_h,p_h)\|_{0,E}^2, \label{eq:indicator1}
\end{align}
we may write the global error estimate as
\begin{align}
  \|\vec{e}_u\|_1^2 + \|e_p\|_0^2 & \leq C_2 \sum_{T \in \mathcal{T}_h} \eta_T^2 . \label{eq:aee-estimate2}
\end{align}

%% file: aee_paper-refinement.tex
\section{Adaptive mesh refinement} \label{sec:refinement}

We use the error indicator~(\ref{eq:indicator1}) to mark elements for refinement. In~\cite[pp.~64--65]{Verfurth-2013-AEE}, two strategies for marking elements are presented. Both require a parameter $\theta \in (0,1)$ and produce a subset $\widetilde{\mathcal{T}}$ of elements marked for refinement. The first strategy, named the \emph{maximum strategy}, is summarized in Algorithm~\ref{alg:maximum}. It marks for refinement all elements for which the error indicator is greater than
or equal to~$\theta$ times the maximum over all elements.

\begin{algorithm}[Maximum Strategy]\label{alg:maximum}
  Given: a partition $\mathcal{T}$, an error indicator $\eta_T$ for each element $T \in \mathcal{T}$, and a threshold $\theta \in (0,1)$. \\
  Sought: a subset $\widetilde{\mathcal{T}}$ of marked elements to be refined.
  \begin{enumerate}
  \item Compute $\eta_{\mathcal{T},\textrm{max}} = \max_{T \in \mathcal{T}} \eta_T$.
  \item Mark all elements $T$ such that $\eta_T \geq \theta {} \eta_{\mathcal{T},\textrm{max}}$ and place in $\widetilde{\mathcal{T}}$.
  \end{enumerate}
\end{algorithm}

The second strategy, named the \emph{equilibration strategy}, marks the elements with the largest error until $\sum_{T \in \widetilde{\mathcal{T}}} \eta_T^2 \geq \theta \sum_{T \in \mathcal{T}} \eta_T^2$. This strategy marks the elements with the largest error until a given proportion of the total error from~(\ref{eq:aee-estimate2}) is attained.

\begin{algorithm}[Equilibration Strategy]\label{alg:equilib}
  Given: a partition $\mathcal{T}$, an error indicator $\eta_T$ for each element $T \in \mathcal{T}$, and a threshold $\theta \in (0,1)$. \\
  Sought: a subset $\widetilde{\mathcal{T}}$ of marked elements to be refined.
  \begin{enumerate}
  \item Compute $\Theta_\mathcal{T} = \sum_{T \in \mathcal{T}} \eta_T^2$ and set $\Sigma_\mathcal{T} = 0$ and $\widetilde{\mathcal{T}} = \emptyset$.
  \item If $\Sigma_{\mathcal{T}} \geq \theta {} \Theta_{\mathcal{T}}$, return $\widetilde{\mathcal{T}}$. Else, go to step 3.
  \item Compute $\widetilde{\eta}_{\mathcal{T},\textrm{max}} = \max_{T \in \mathcal{T} \setminus \widetilde{\mathcal{T}}} \eta_T$.
  \item For each element $T$ in $\mathcal{T} \setminus \widetilde{\mathcal{T}}$ such that $\eta_T = \widetilde{\eta}_{\mathcal{T},\textrm{max}}$
    \begin{enumerate}
    \item Update $\Sigma_\mathcal{T} \leftarrow \Sigma_\mathcal{T} + \eta_T^2$
    \item Add $T$ to $\widetilde{\mathcal{T}}$
    \end{enumerate}
  \item Go to step 2.
  \end{enumerate}
\end{algorithm}

There may be cases where the errors are concentrated in only a few elements, with the remaining contributing very little to the total error. In these cases, we should choose a small $\epsilon \in [0,1]$ and always mark for refinement the $\epsilon |\mathcal{T}|$ elements of largest error and then apply either the maximum strategy or the equilibration strategy to the remaining $(1-\epsilon)|\mathcal{T}|$ elements.

%% file: aee_paper-numerical.tex
\section{Numerical experiments} \label{sec:numerical}

In this section, we test the effectiveness of the error estimator~(\ref{eq:indicator1}) to drive an adaptive mesh refinement process. We perform three experiments. The first experiment is on a nonconvex 2D domain, the second is on a 2D domain with discontinuous inflow and an obstacle, and the third is on a nonconvex 3D domain. In all experiments, we use the maximum~(Algorithm~\ref{alg:maximum}) and equilibration~(Algorithm~\ref{alg:equilib}) strategies with parameters $(\epsilon,\theta) \in \{0,0.001,0.01\} \times \{0.25,0.5,0.75\}$. We compare results with the uniform refinement strategy (i.e., marking all elements for refinement) to show that the adaptive procedure successfully yields a more accurate solution with fewer degrees of freedom with the error given by~(\ref{eq:aee-estimate2}). For simplicity, we assume the constant factor in the estimate is $1$ as its exact value is irrelevant when comparing estimates from successive meshes. We used a computer with two 8-core 2.10 GHz CPUs\ with 1 TB\ of memory running Linux openSUSE 42.3 and \textsc{Matlab} version {9.2.0.538062 (R2017a)}. We implemented the error estimates in \textsc{Matlab}, generated and refined the meshes with Netgen~\cite{Schoberl-1997} version 6.2, and visualized the solutions using Paraview~\cite{Paraview} version 5.1.2. The underlying linear systems were solved using the sparse direct solver (backslash) within \textsc{Matlab}.

\subsection{2D nonconvex experiment}

For the 2D nonconvex numerical experiment, we consider the nonconvex domain depicted in Figure~\ref{fig:domain}. The domain is partitioned into three regions: one Stokes region in which the fluid free-flows ($K^{-1} = 0$) and two Darcy regions with permeability tensors $K = 5 \cdot 10^{-4} I$ and $K = 5 \cdot 10^{-2} I$, where $I$ denotes the identity matrix. The fluid flows from left to right, with parabolic inflow at the left end of the first Darcy region and a do-nothing outflow at the right end of this region. No flow is allowed through any other boundary edge. The Stokes and second Darcy regions appear as pockets at the top and bottom of the domain. The fluid has a constant viscosity $\mu = \mu^* = 10^{-3}$ throughout. We used zero right-hand sides $\vec{f} = 0$ and $g = 0$ throughout the domain.

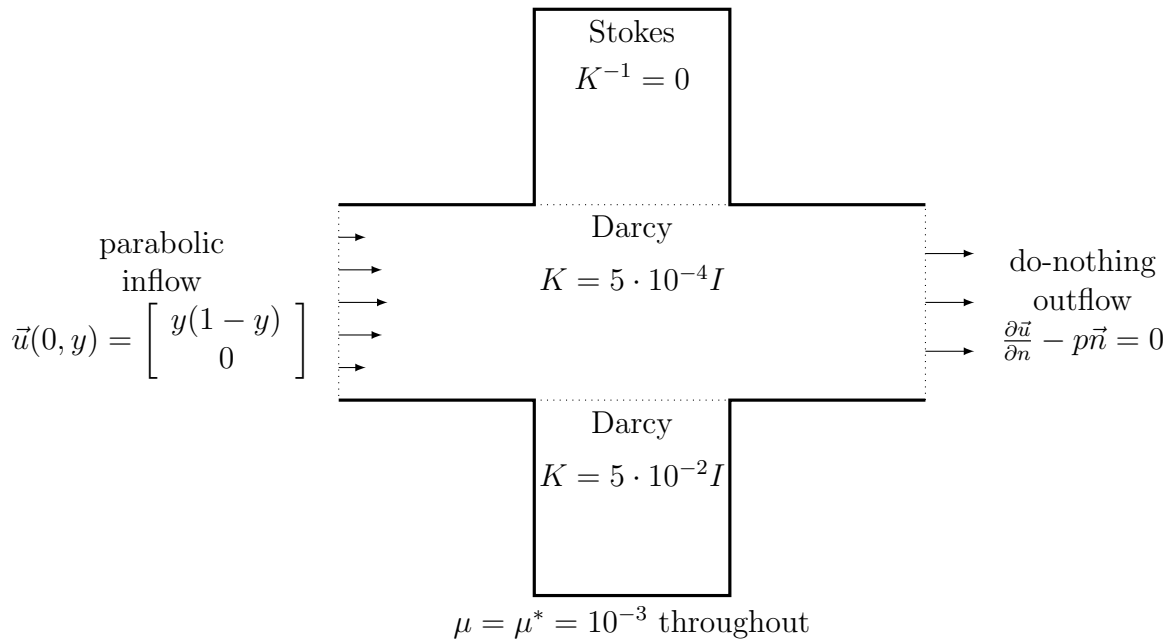
\begin{figure}[h!]
  \centering
  \begin{tikzpicture}[scale=2.6,>=latex]
    \draw[very thick] (0,0) -- (1,0) -- (1,-1) -- (2,-1) -- (2,0) -- (3,0);
    \draw[dotted] (3,0) -- (3,1); 
    \draw[very thick] (3,1) -- (2,1) -- (2,2) -- (1,2) -- (1,1) -- (0,1);
    \draw[dotted] (0,1) -- (0,0); 
    \draw[dotted] (1,1) -- (2,1); 
    \draw[dotted] (1,0) -- (2,0); 

    \node[anchor=north] at (1.5,2) (stokes) {Stokes};
    \node[anchor=north] at (stokes.south) {$K^{-1} = 0$};
    \node[anchor=north] at (1.5,1) (darcy1) {Darcy};
    \node[anchor=north] at (darcy1.south) {$K = 5 \cdot 10^{-4} I$};
    \node[anchor=north] at (1.5,0) (darcy2) {Darcy};
    \node[anchor=north] at (darcy2.south) {$K = 5 \cdot 10^{-2} I$};


    \draw[->] (0,0.166667) -- (0.138889,0.166667);
    \draw[->] (0,0.333333) -- (0.222222,0.333333);
    \draw[->] (0,0.500000) -- (0.250000,0.500000);
    \draw[->] (0,0.666667) -- (0.222222,0.666667);
    \draw[->] (0,0.833333) -- (0.138889,0.833333);
    \node[anchor=east] at (0,0.5) {\begin{tabular}{c}
        parabolic \\
        inflow \\
        $\vec{u}(0,y) = \left[\begin{array}{c}
          y(1-y) \\
          0
        \end{array}\right]$
    \end{tabular} };

    \draw[->] (3,0.75) -- (3.25,0.75);
    \draw[->] (3,0.5) -- (3.25,0.5);
    \draw[->] (3,0.25) -- (3.25,0.25);
    \node[anchor=west] at (3.25,0.5) {\begin{tabular}{c}
        do-nothing \\
        outflow \\
        $\frac{\partial \vec{u}}{\partial n} - p\vec{n} = 0$
    \end{tabular} };

    \node[anchor=north] at (1.5,-1) {$\mu = \mu^* = 10^{-3}$ throughout};
  \end{tikzpicture}
  \caption{Nonconvex 2D domain used for numerical experiments.}
  \label{fig:domain}
\end{figure}

We tested the maximum~(Algorithm~\ref{alg:maximum}) and equilibration~(Algorithm~\ref{alg:equilib}) strategies on this domain with all combinations of $(\epsilon,\theta) \in \{0,0.001,0.01\} \times \{0.25,0.5,0.75\}$. There were 1431 degrees of freedom in the initial mesh. For each choice of adaptive strategy, we refined the mesh 10 times, and, for comparison, we uniformly refined the mesh 5 times.

Figure~\ref{fig:2d-adapt-vs-uniform} compares the numbers of degrees of freedom against the computed error estimates for each experiment. The degrees of freedom are plotted on the x-axis on a log scale; the error estimates are plotted on the y-axis also on a log scale. There are 9 subplots arranged on a 3 x 3 grid. The rows correspond to the choices of $\epsilon$ with the top row $\epsilon = 0$, the middle $\epsilon = 0.001$, and the bottom $\epsilon = 0.01$. The columns correspond to the choices of $\theta$ with the left column $\theta = 0.25$, the middle column $\theta = 0.5$, and the right column $\theta = 0.75$. Each plot has the results from the uniform refinement as the dashed black line, along with the results from the maximum (red + line) and equilibration (blue dotted line) strategies. We can observe that the error decreases as the degrees of freedom increase, as we expect. However, for each of the adaptive strategies, we attain lower errors for a given size of the problem. This suggests that the adaptive strategies yield better accuracy with less work.

\begin{figure}[h!]
  \centering
  \includegraphics[width=18cm]{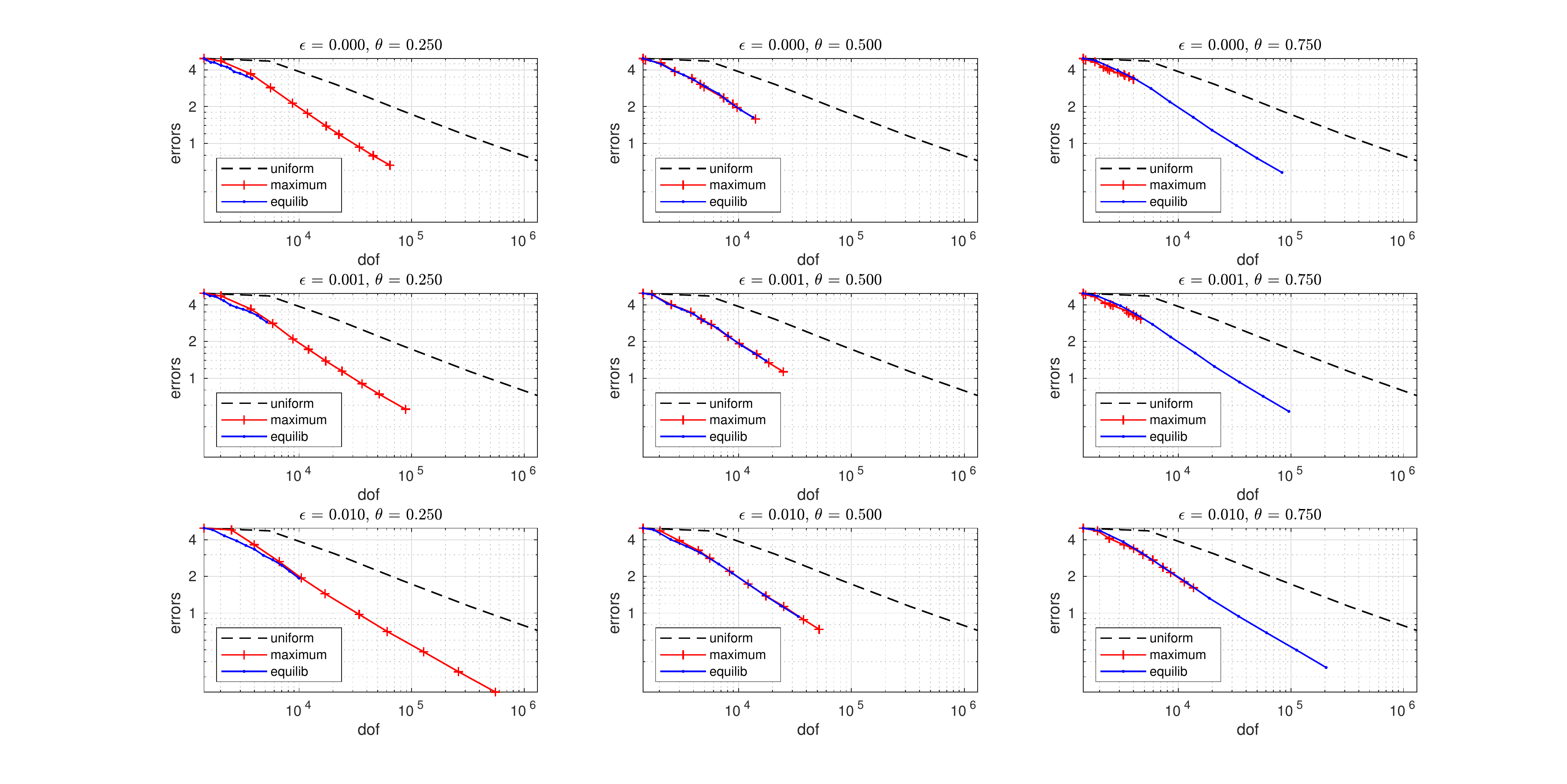}
  \caption{Comparison of each adaptive mesh refinement strategy against uniform refinement for the nonconvex 2D experiment.}
  \label{fig:2d-adapt-vs-uniform}
\end{figure}

Figure~\ref{fig:2d-comp-epsilon} shows how the choice of $\epsilon$ affects the performance of the adaptive strategy. There are 6 plots, each fixing a choice of strategy (maximum or equilibration) and a choice of $\theta$ but varying $\epsilon$. The plots are arranged in a 2x3 grid. The top row displays results with the maximum strategy and the bottom displays results with the equilibration strategy; the first column displays results for $\theta = 0.25$, the middle for $\theta = 0.5$, and the final for $\theta = 0.75$. For each plot, the black (o) line is $\epsilon = 0$, the red (+) line is $\epsilon = 0.001$, and the blue dotted line is $\epsilon = 0.01$. As before, the degrees of freedom are plotted on the x-axis on a log scale with the computed errors on the y-axis on a log scale. The plots show that the choice of $\epsilon$ has very little effect on the accuracy of the computation per degree of freedom. However, there is a significant difference in the growth in degrees of freedom per iteration of refinement. Choosing a small value of $\epsilon$ causes the degrees of freedom to grow slowly, as fewer elements are refined upfront, which would require more iterations of refinement to reach a desired error tolerance. On the other hand, choosing a large value of $\epsilon$ might cause the degrees of freedom to grow too rapidly, resulting in a larger problem than necessary for a chosen error tolerance.

\begin{figure}[h!]
  \centering
  \includegraphics[width=18cm]{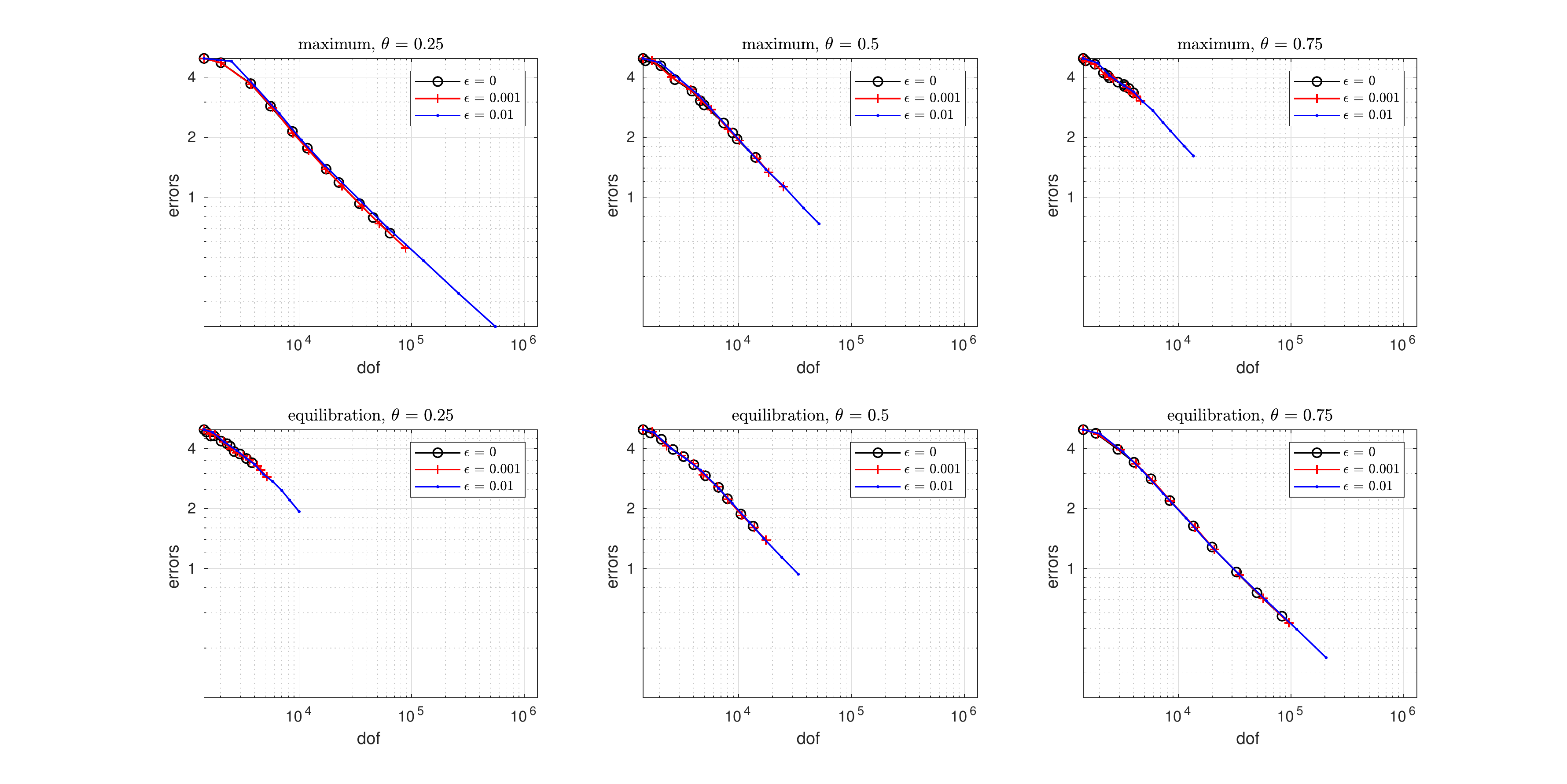}
  \caption{Comparison of varying choices of $\epsilon$ for each adaptive strategy for the nonconvex 2D experiment.}
  \label{fig:2d-comp-epsilon}
\end{figure}

Figure~\ref{fig:2d-comp-theta} shows how the choice of $\theta$ affects the performance of the adaptive strategy. The setup is similar to that of Figure~\ref{fig:2d-comp-epsilon} except that, now, $\epsilon$ is fixed per plot and $\theta$ varies. The first and second rows display results for the maximum and equilibration strategies, respectively; the first column displays results for $\epsilon = 0$, the second $\epsilon = 0.001$, and the final $\epsilon = 0.01$. As with varying $\epsilon$, varying $\theta$ has very little effect on the accuracy of the computation per degree of freedom. However, the choice of $\theta$ does have an effect on the growth of the degrees of freedom. For the maximum strategy, a smaller value of $\theta$ results in a larger growth in degrees of freedom. This is expected as a smaller value of $\theta$ results in more elements being marked for refinement. For the equilibration strategy, a smaller value of $\theta$ results in a smaller growth, as expected.

\begin{figure}[h!]
  \centering
  \includegraphics[width=18cm]{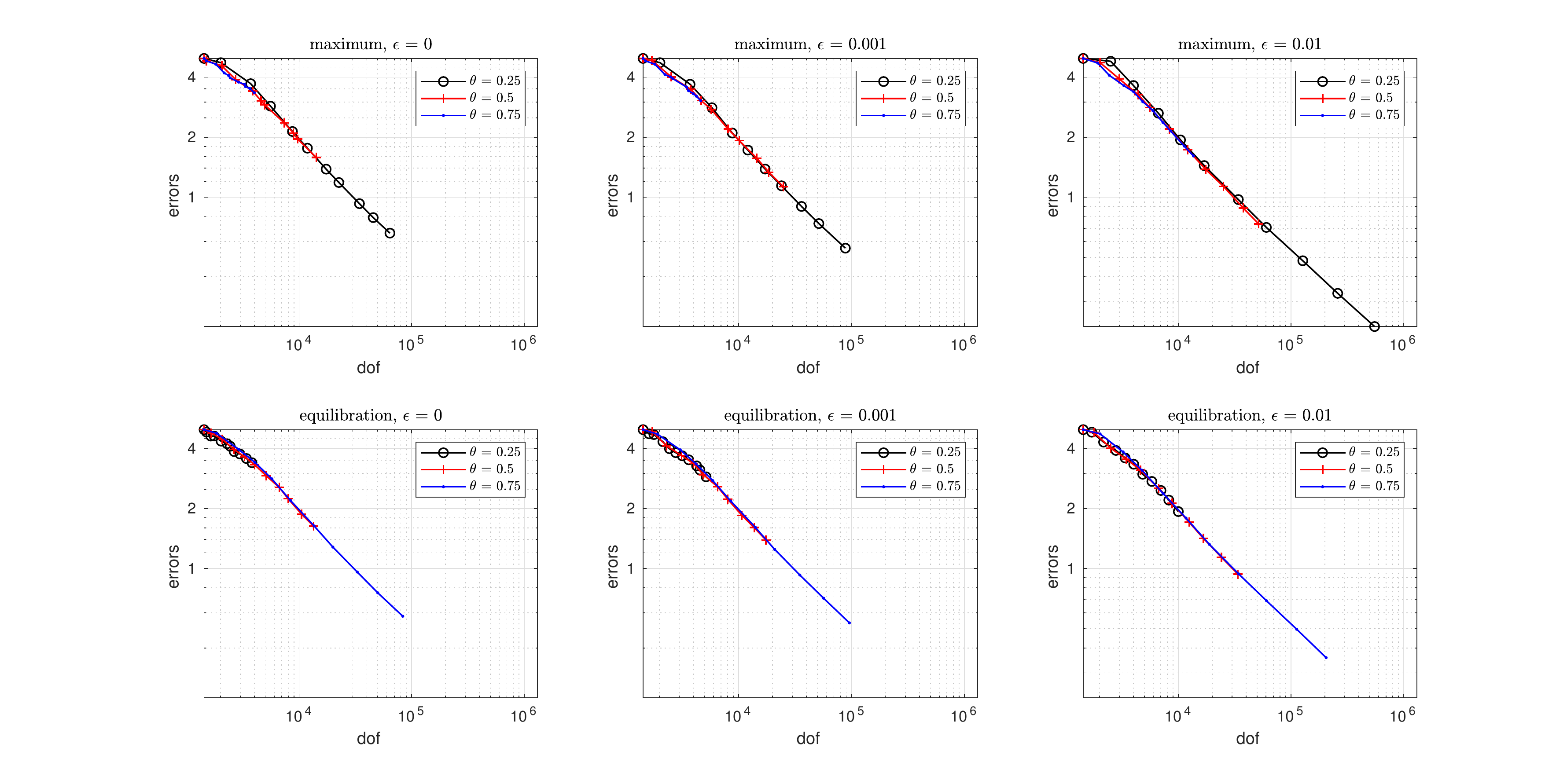}
  \caption{Comparison of varying choices of $\theta$ for each adaptive strategy for the nonconvex 2D experiment.}
  \label{fig:2d-comp-theta}
\end{figure}

The mesh for the equilibration strategy with $\epsilon = 0.01$ and $\theta = 0.25$ at various stages of refinement is shown in Figure~\ref{fig:equilib-2d-01-25-meshes}. The top left corner shows the initial mesh, the top right after 1 iteration of refinement, the bottom left after 5 iterations, and the bottom right after 10 iterations. As expected, the error estimator is able to identify the non-convex corners of the domain. Perhaps unexpectedly, the estimator suggests further refinement along the no-flow boundary of the middle (Darcy) region. The flow for the final iteration of this strategy is visualized in Figure~\ref{fig:equilib-2d-01-25-flow}. The domain is colored according to $K^{-1}$ on a log scale to accentuate the difference in permeabilities between the three regions. The flow is visualized as streamlines flowing left to right colored by the pressure. Note that the pressure decreases as the flow moves from the left end of the domain to the right. Also note that the flow tends to the higher permeability Stokes and lower Darcy regions.

\begin{figure}[h!]
  \centering
  \begin{tabular}{cc}
    \includegraphics[width=4.5cm]{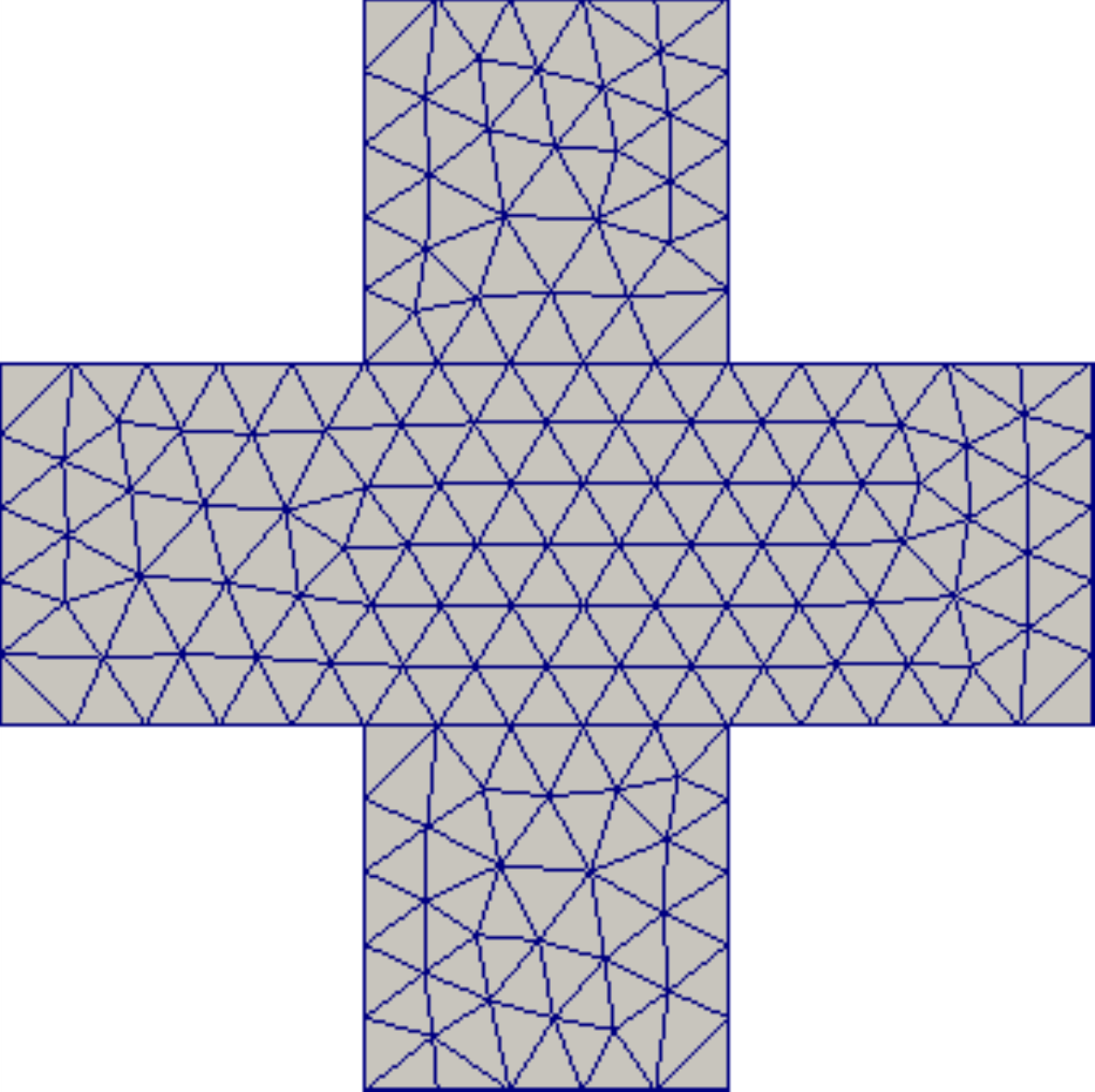} & \includegraphics[width=4.5cm]{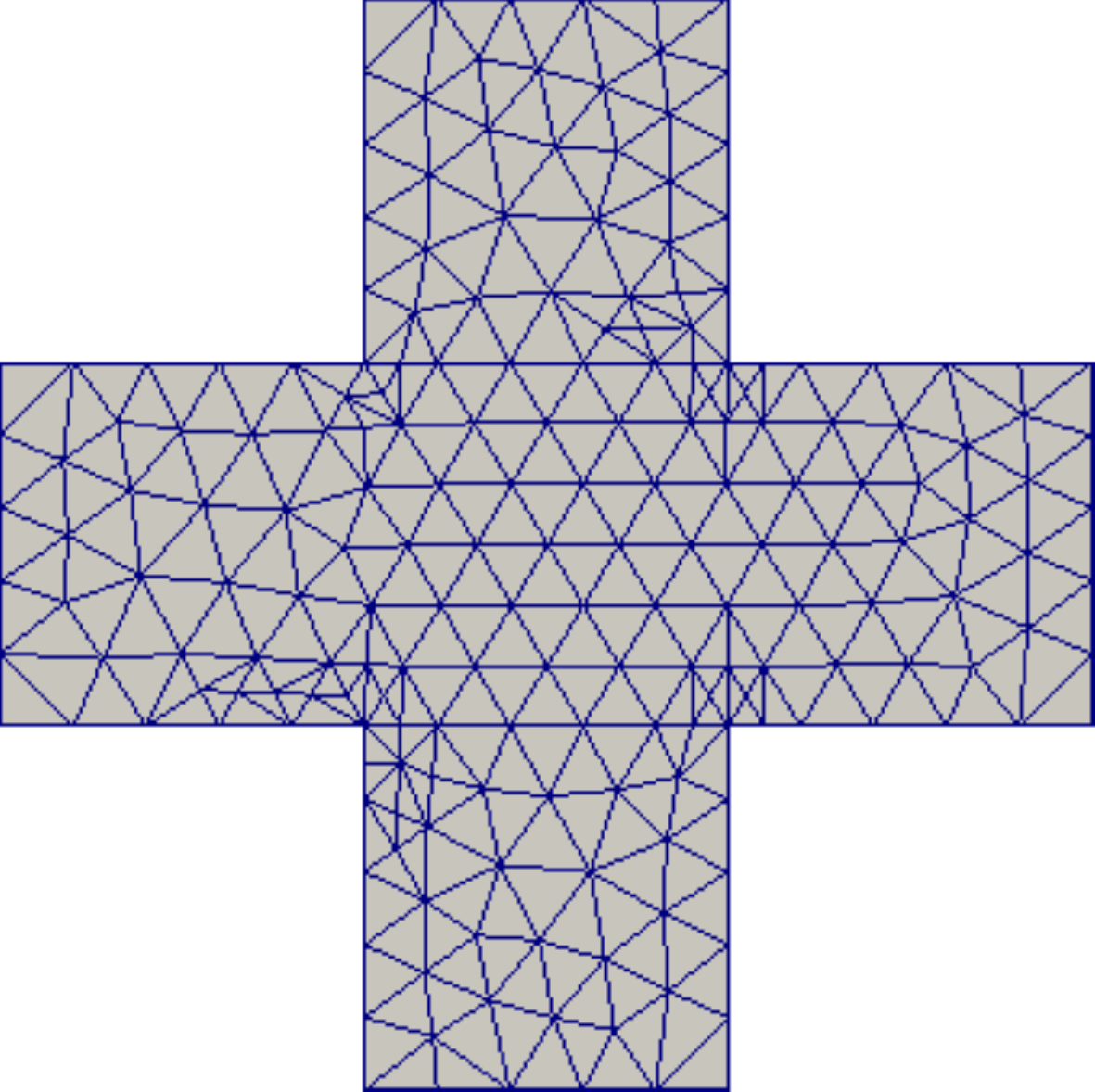} \\
    \includegraphics[width=4.5cm]{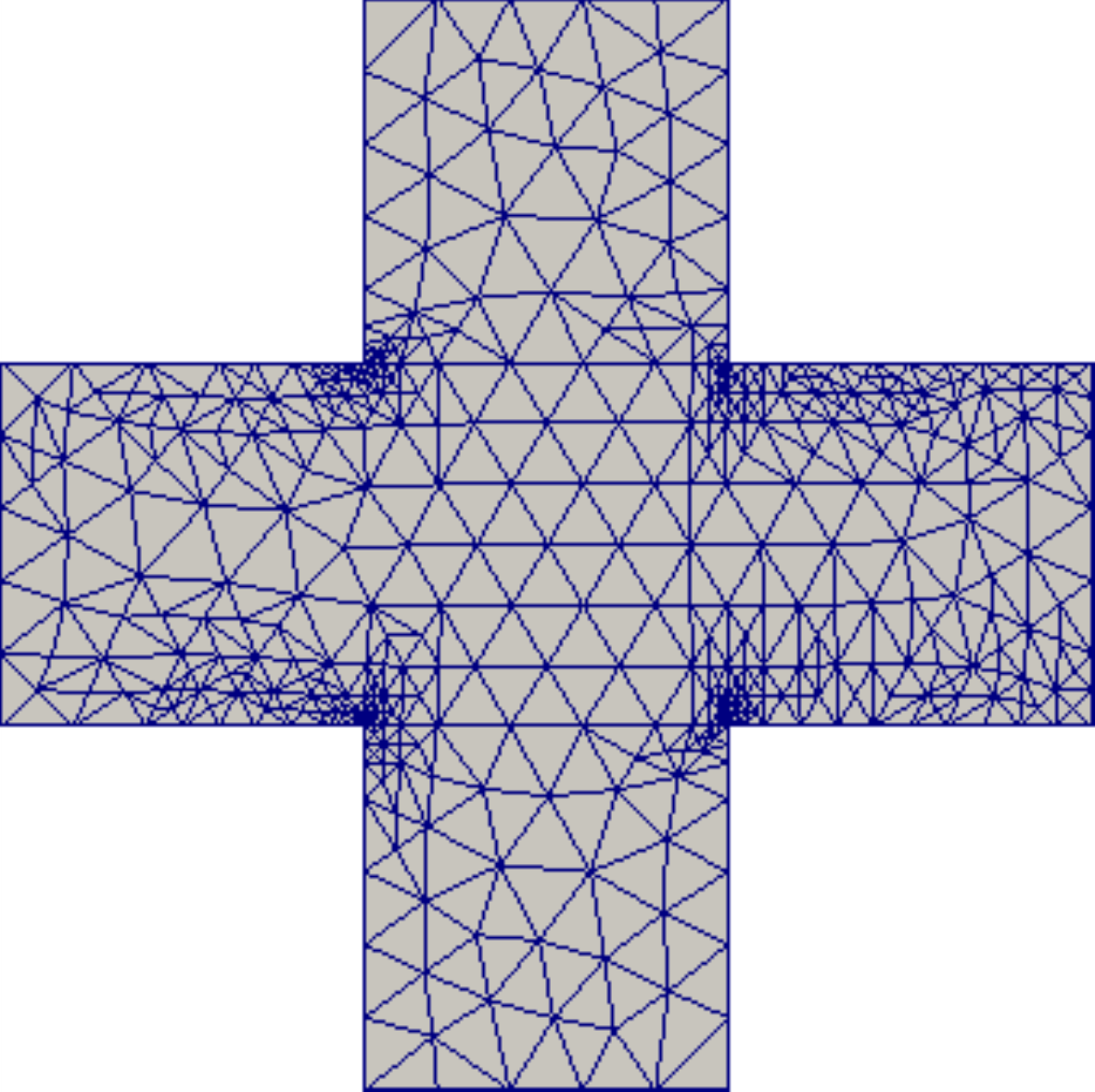} & \includegraphics[width=4.5cm]{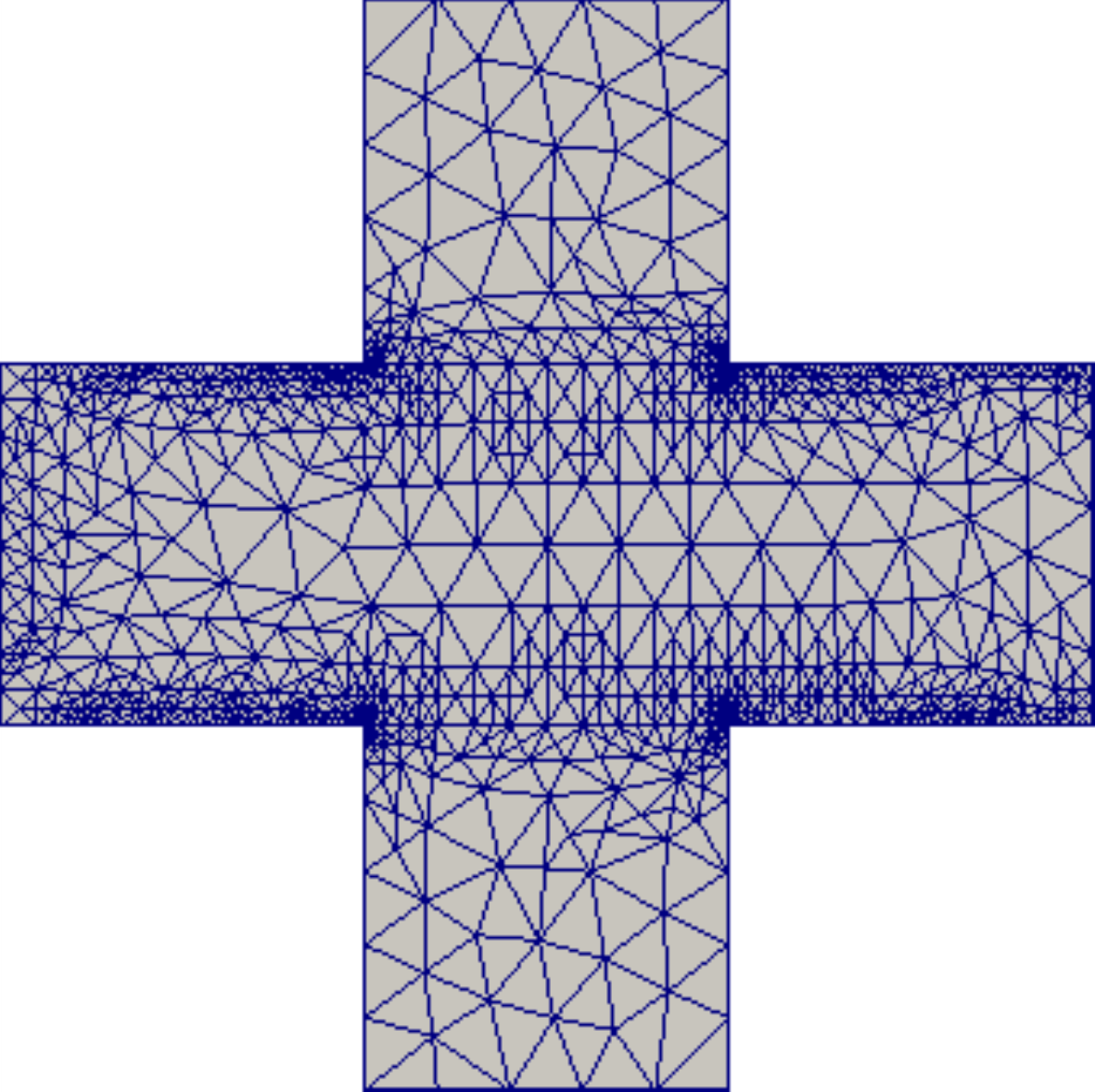} \\
  \end{tabular}
  \caption{Meshes at various stages of refinement for the nonconvex 2D problem using an equilibration strategy with $\epsilon = 0.01$ and $\theta = 0.25$. The top left corner shows the initial mesh, the top right after 1 iteration of refinement, the bottom left after 5 iterations, and the bottom right after 10 iterations.}
  \label{fig:equilib-2d-01-25-meshes}
\end{figure}

\begin{figure}[h!]
  \centering
  \includegraphics[width=18.5cm]{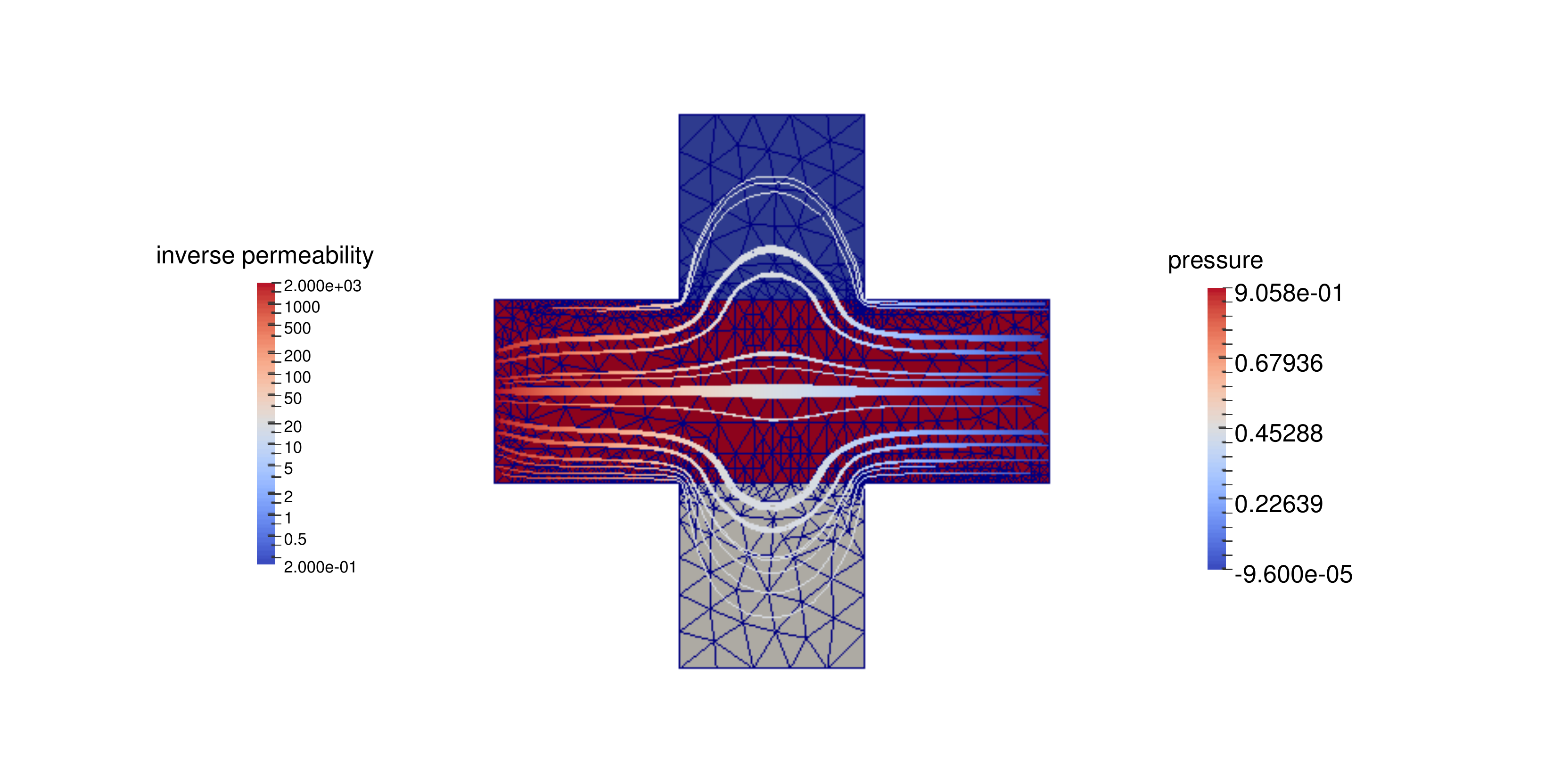}
  \caption{Visualization of the flow for the 2D nonconvex problem using an equilibration strategy with $\epsilon = 0.01$ and $\theta = 0.25$ after the final iteration of refinement. The domain is colored according to $K^{-1}$ on a log scale to accentuate the difference in permeabilities between the three regions. The flow is visualized as streamlines flowing left to right colored by the pressure.}
  \label{fig:equilib-2d-01-25-flow}
\end{figure}

\subsection{2D obstacle experiment}

For the second 2D experiment, we consider the square domain depicted in Figure~\ref{fig:domain-obstacle}. For the boundary conditions, we use a discontinuous inflow with a constant velocity of $[1/4, 0]^T$ and a do-nothing outflow. In the center of the domain is an obstacle through which no flow is possible. Throughout most of the domain there is porous material with permeability tensor $K = 5 \cdot 10^{-4} I$. Near the obstacle, however, are Stokes and Darcy regions with permeability tensors satisfying $K^{-1} = 0$ and $K = 5 \cdot 10^{-2} I$, respectively. As with the previous experiment, we use a constant viscosity of $\mu = \mu^*= 10^{-3}$. We used zero right-hand sides $\vec{f} = 0$ and $g = 0$ throughout the domain.

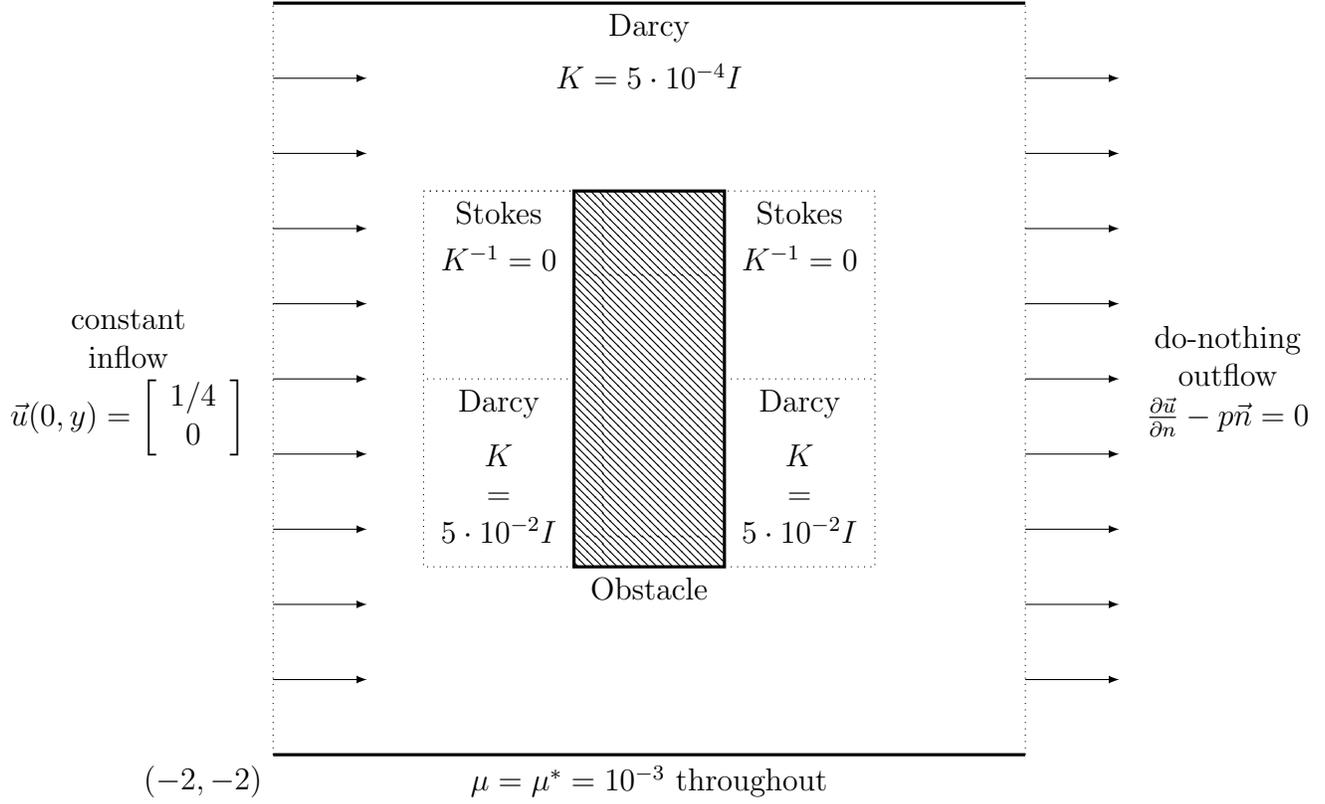
\begin{figure}[h!]
  \centering
  \begin{tikzpicture}[scale=2.5,>=latex]
    \draw[very thick] (-2,-2) -- ( 2,-2);
    \draw[very thick] (-2, 2) -- ( 2, 2);
    \draw[dotted] (-2,-2) -- (-2, 2);
    \draw[dotted] ( 2,-2) -- ( 2, 2);
    \draw[very thick, pattern=north west lines] (-0.4,-1) rectangle (0.4, 1);
    \draw[dotted] (-0.4,-1) -- (-1.2,-1) -- (-1.2, 1) -- (-0.4, 1);
    \draw[dotted] ( 0.4,-1) -- ( 1.2,-1) -- ( 1.2, 1) -- (-1.2, 1);
    \draw[dotted] (-0.4, 0) -- (-1.2, 0);
    \draw[dotted] ( 0.4, 0) -- ( 1.2, 0);

    \node[anchor=north] at (0,-1) (obstacle) {Obstacle};
    \node[anchor=north] at (0,2) (darcy1) {Darcy};
    \node[anchor=north] at (darcy1.south) {$K = 5 \cdot 10^{-4} I$};
    \node[anchor=north] at (-0.8, 1) (stokes1) {Stokes};
    \node[anchor=north] at (stokes1.south) {$K^{-1} = 0$};
    \node[anchor=north] at ( 0.8, 1) (stokes2) {Stokes};
    \node[anchor=north] at (stokes2.south) {$K^{-1} = 0$};
    \node[anchor=north] at (-0.8, 0) (darcy2) {Darcy};
    \node[anchor=north] at (darcy2.south) {$\begin{array}{c}
      K \\
      = \\
      5 \cdot 10^{-2} I
    \end{array}$};
    \node[anchor=north] at ( 0.8, 0) (darcy3) {Darcy};
    \node[anchor=north] at (darcy3.south) {$\begin{array}{c}
      K \\
      = \\
      5 \cdot 10^{-2} I
    \end{array}$};

    \draw[->] (-2,1.6) -- (-1.5,1.6);
    \draw[->] (-2,1.2) -- (-1.5,1.2);
    \draw[->] (-2,0.8) -- (-1.5,0.8);
    \draw[->] (-2,0.4) -- (-1.5,0.4);
    \draw[->] (-2,0) -- (-1.5,0);
    \draw[->] (-2,-0.4) -- (-1.5,-0.4);
    \draw[->] (-2,-0.8) -- (-1.5,-0.8);
    \draw[->] (-2,-1.2) -- (-1.5,-1.2);
    \draw[->] (-2,-1.6) -- (-1.5,-1.6);
    \node[anchor=east] at (-2,0) {\begin{tabular}{c}
      constant \\
      inflow \\
      $\vec{u}(0,y) = \left[\begin{array}{c}
        1/4 \\
        0
      \end{array} \right]$
    \end{tabular}};
    \draw[->] (2,1.6) -- (2.5,1.6);
    \draw[->] (2,1.2) -- (2.5,1.2);
    \draw[->] (2,0.8) -- (2.5,0.8);
    \draw[->] (2,0.4) -- (2.5,0.4);
    \draw[->] (2,0) -- (2.5,0);
    \draw[->] (2,-0.4) -- (2.5,-0.4);
    \draw[->] (2,-0.8) -- (2.5,-0.8);
    \draw[->] (2,-1.2) -- (2.5,-1.2);
    \draw[->] (2,-1.6) -- (2.5,-1.6);
    \node[anchor=west] at (2.5,0) {\begin{tabular}{c}
        do-nothing \\
        outflow \\
        $\frac{\partial \vec{u}}{\partial n} - p\vec{n} = 0$
    \end{tabular} };

  \node[anchor=north] at (0,-2) {$\mu = \mu^* = 10^{-3}$ throughout};

  \node[anchor=north east] at (-2,-2) {$(-2,-2)$};
  \node[anchor=south west] at (2,2) {$(2,2)$};
  \end{tikzpicture}
  \caption{Domain in 2D with obstacle used for numerical experiments.}
  \label{fig:domain-obstacle}
\end{figure}

As with the previous experiment, we tested the maximum and equilibration strategies with the same combinations of $(\epsilon,\theta) \in \{0,0.001,0.01\} \times \{0.25,0.5,0.75\}$. There were 3978 degrees of freedom in the initial mesh. For each choice of adaptive strategy, we refined the mesh 10 times, and, for comparison, we uniformly refined the mesh 5 times.

Figure~\ref{fig:2d-obstacle-vs-uniform} compares the degrees of freedom against the computed error estimate for each experiment, in a manner similar to Figure~\ref{fig:2d-adapt-vs-uniform} from the previous section. Note that the results are similar to the previous experiment with all adaptive strategies outperforming uniform refinement and very little differences in performance aside from the growth in degrees of freedom.

\begin{figure}[h!]
  \centering
  \includegraphics[width=18cm]{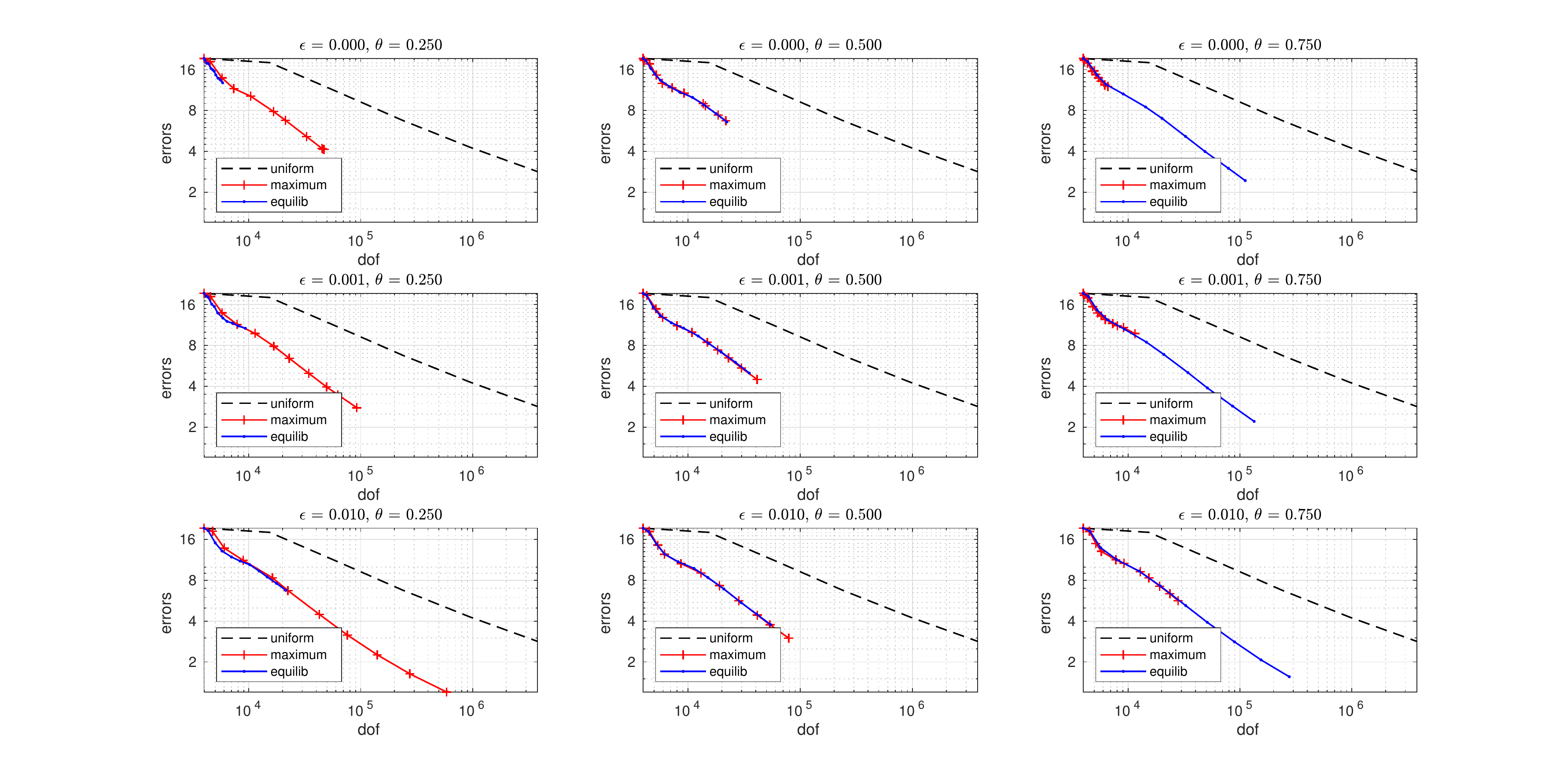}
  \caption{Comparison of each adaptive mesh refinement strategy against uniform refinement for the 2D obstacle experiment.}
  \label{fig:2d-obstacle-vs-uniform}
\end{figure}

Figure~\ref{fig:equilib-obstacle-01-25-meshes} presents the meshes for various levels of refinement using an equilibration strategy with $\epsilon = 0.01$ and $\theta = 0.25$. The top left corner displays the initial mesh used. The mesh after a single step of refinement is depicted in the top right corner. Note that the only refinement was performed at the corners of the obstacle, where most of the error was centered. The mesh in the lower left corner is after 5 refinement iterations. At this point, the errors at the corners of the obstacle were reduced sufficiently so that errors in other parts of the domain, notably the boundary walls at the top and bottom of the domain along with the interface between subregions near the obstcle corners, could be identified and refined. The lower right picture is of the mesh after the final (10th) iteration of refinement. At this stage, the errors are located in the same regions as before, suggesting a good approximation throughout much of the rest of the domain.

\begin{figure}[h!]
  \centering
  \begin{tabular}{cc}
    \includegraphics[width=4.5cm]{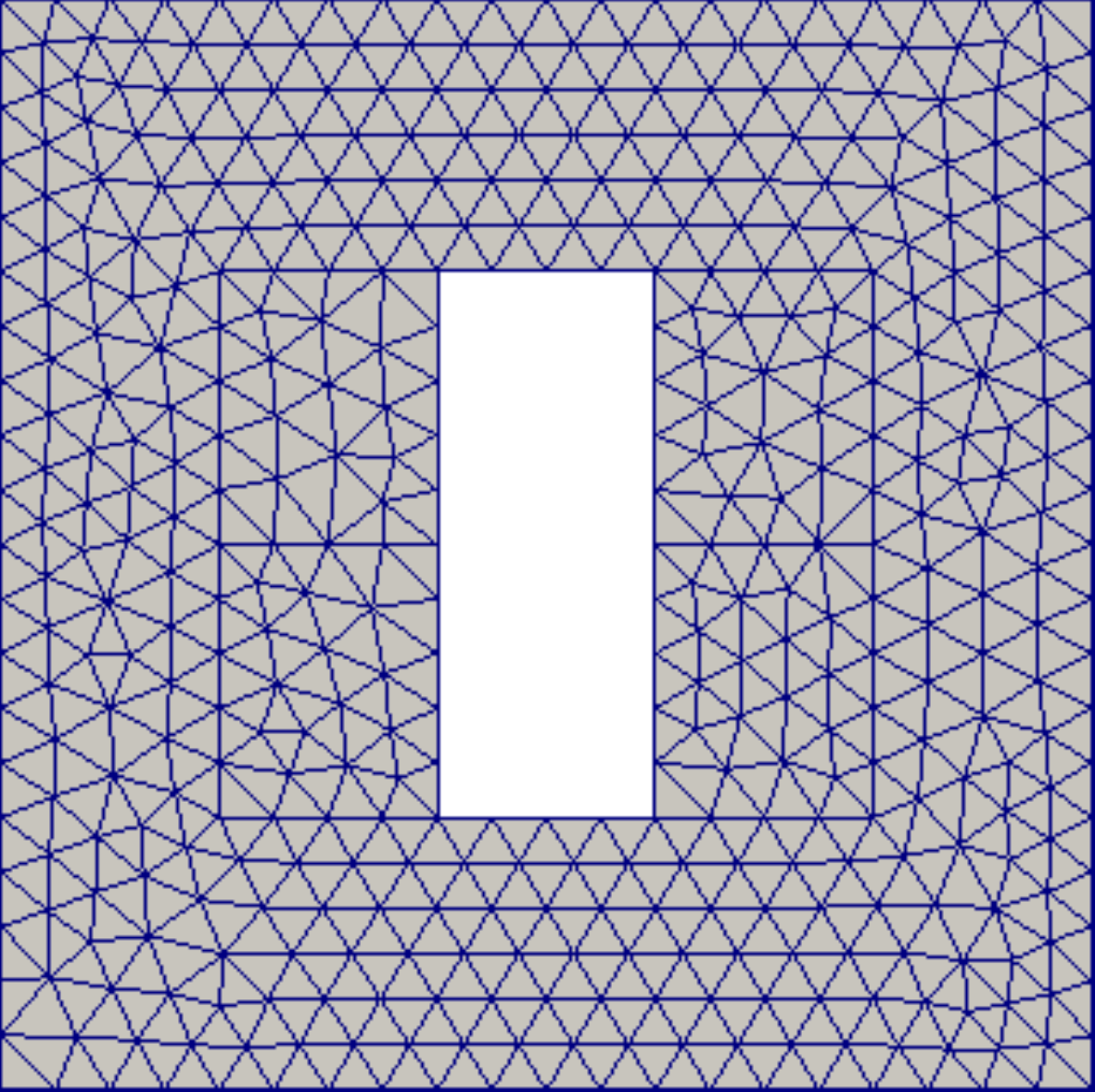} & \includegraphics[width=4.5cm]{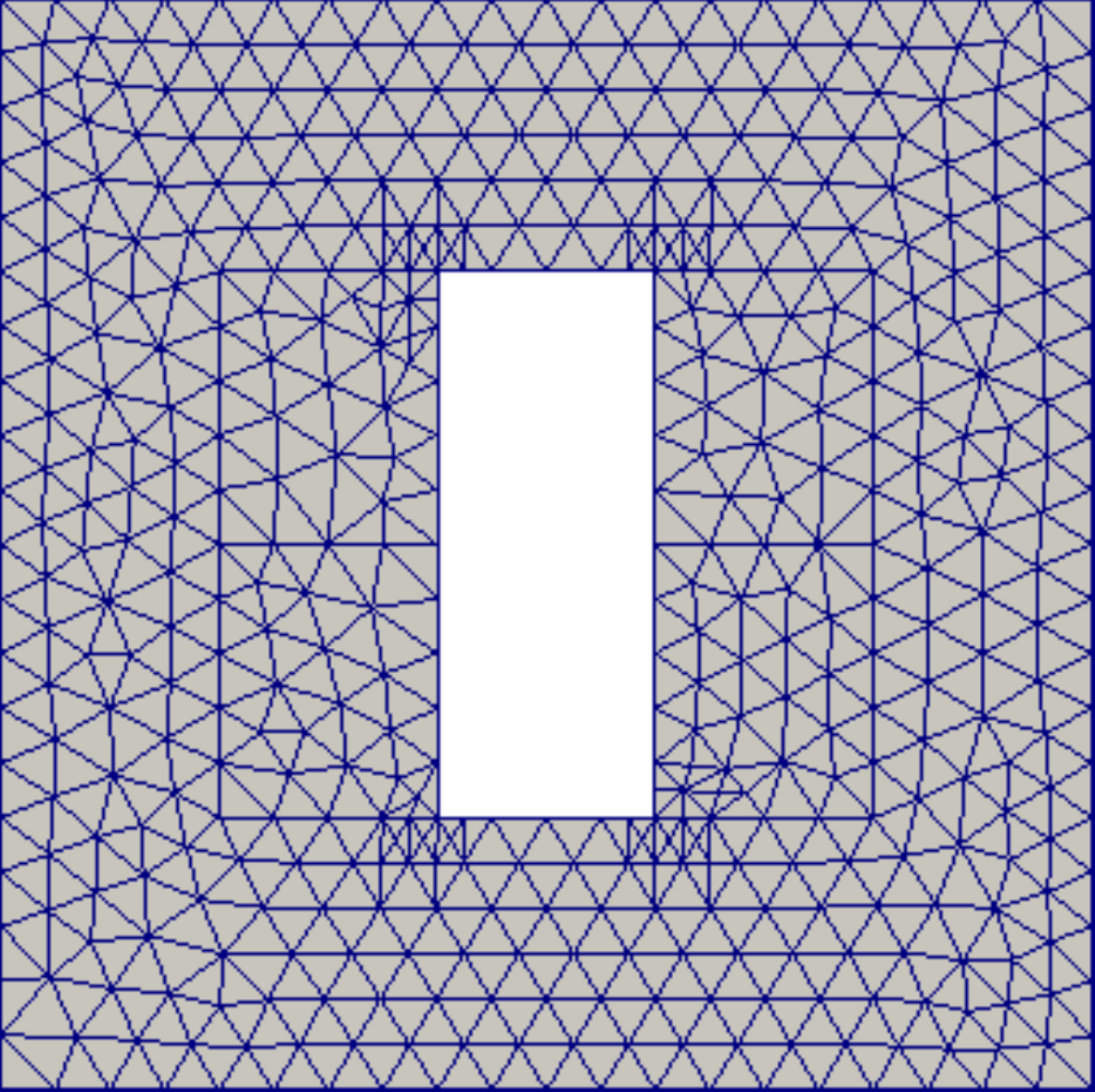} \\
    \includegraphics[width=4.5cm]{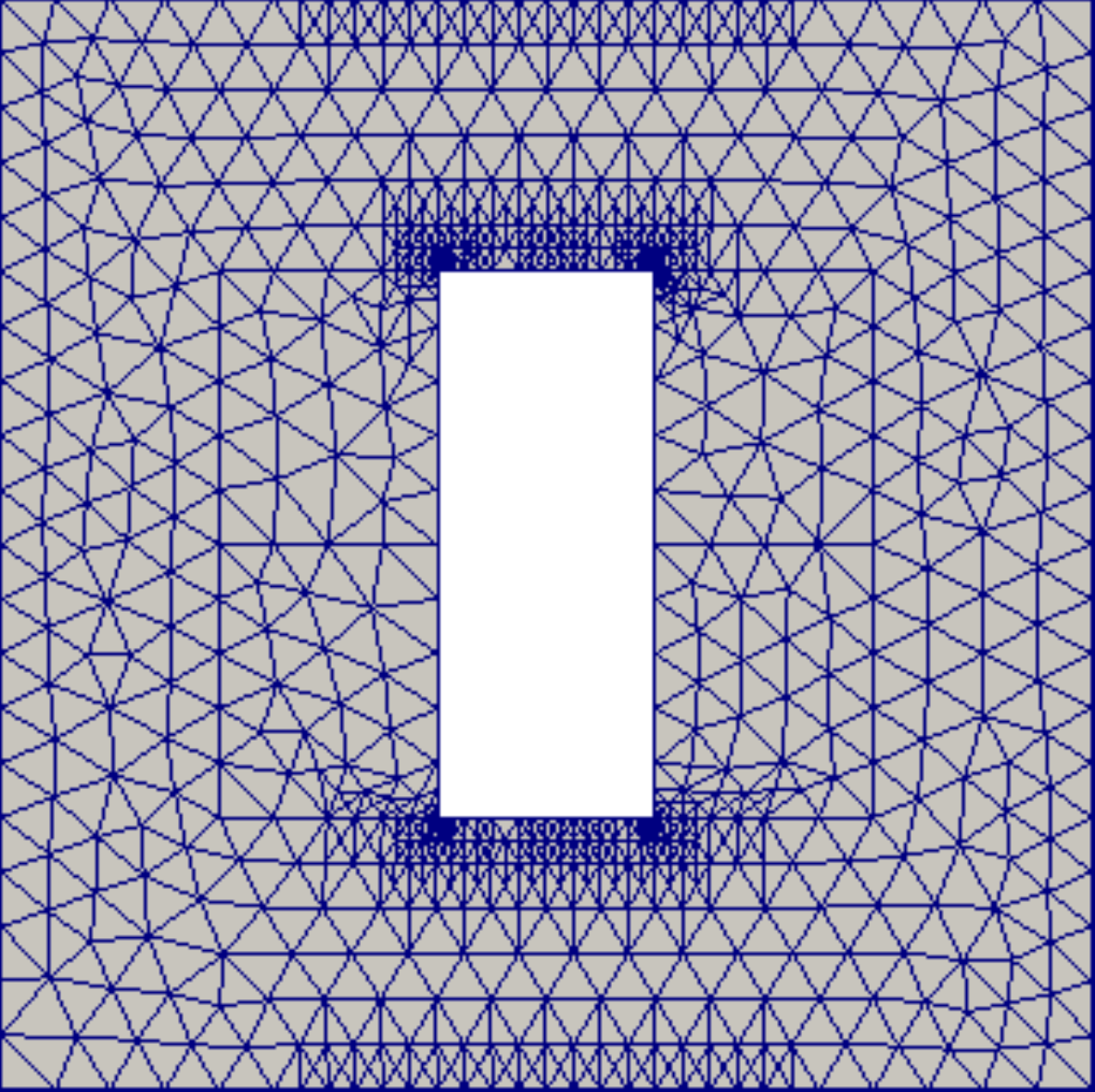} & \includegraphics[width=4.5cm]{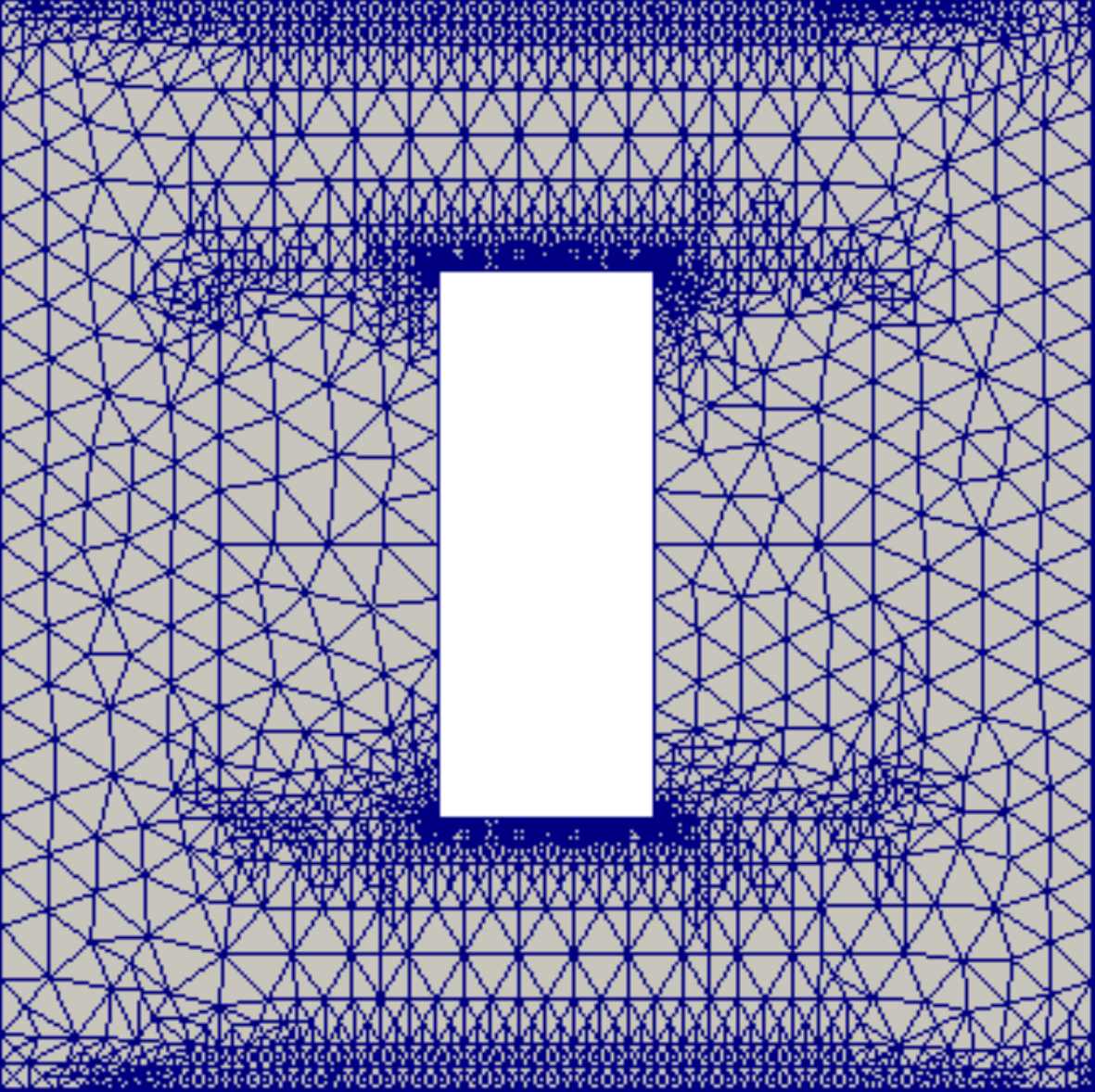} \\
  \end{tabular}
  \caption{Meshes at various stages of refinement for the obstacle problem using an equilibration strategy with $\epsilon = 0.01$ and $\theta = 0.25$. The top left corner shows the initial mesh, the top right after 1 iteration of refinement, the bottom left after 5 iterations, and the bottom right after 10 iterations.}
  \label{fig:equilib-obstacle-01-25-meshes}
\end{figure}

Figure~\ref{fig:equilib-obstacle-01-25-flow} depicts the flow from the solution to the final (10th) iteration of an equilibration strategy with $\epsilon = 0.01$ and $\theta = 0.25$. As expected, the pressure decreases from inflow to outflow, and the flow avoids the obstacle in the center of the domain. The subregions are colored according to the inverse permeability on a log scale with the Stokes (free-flow) region in blue, the high permeability Darcy region in gray, and the low permeability Darcy region in red. As we would expect, the flow tends to the regions of higher permeability, as evidenced by the bulges at the corners of these regions. However, the flow appears to quickly leave the regions due to encountering the obstacle.

\begin{figure}[h!]
  \centering
  \includegraphics[width=17cm]{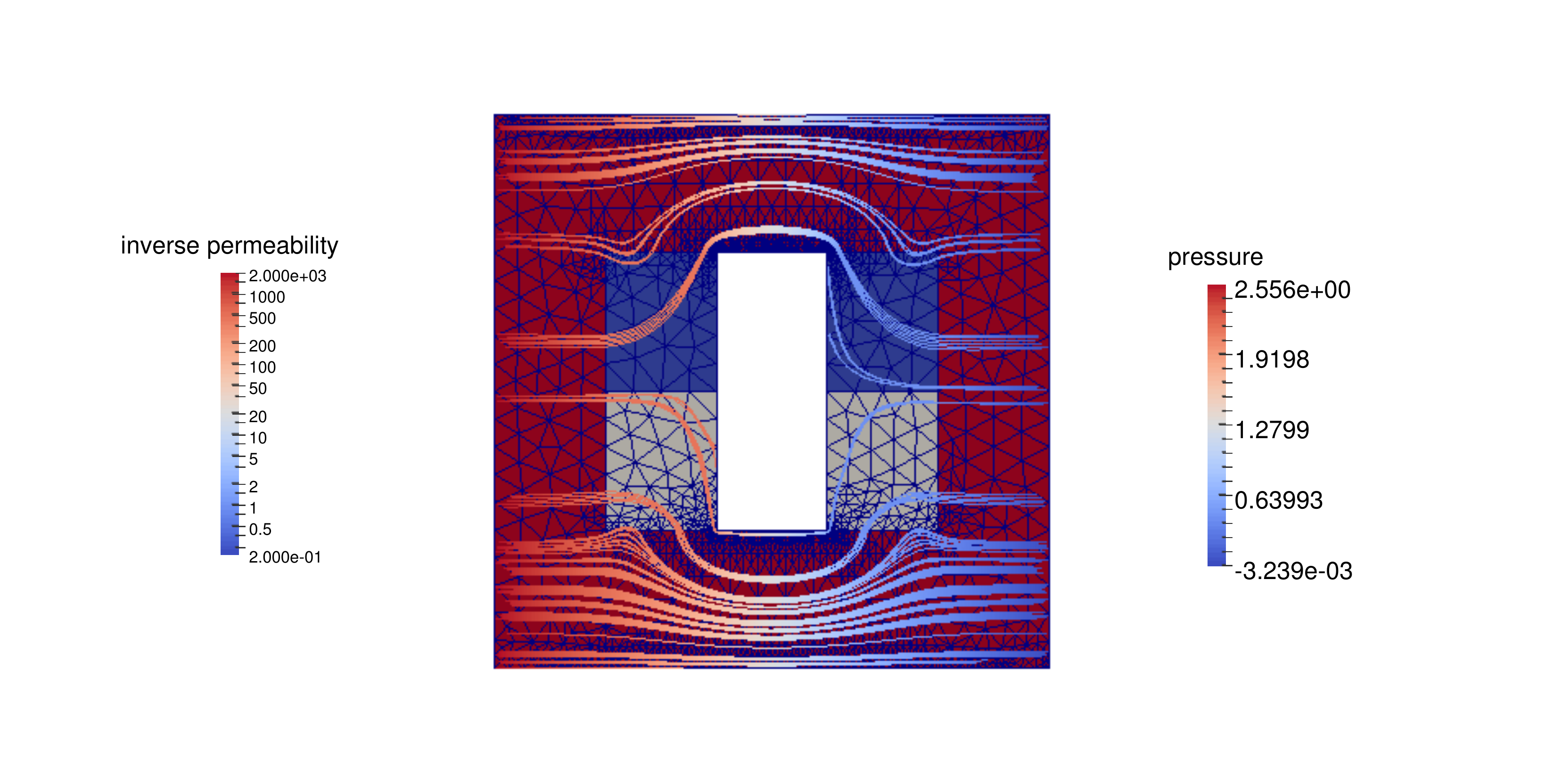}
  \caption{Visualization of the flow for the 2D obstacle problem using an equilibration strategy with $\epsilon = 0.01$ and $\theta = 0.25$ after the final iteration of refinement. The domain is colored according to $K^{-1}$ on a log scale to accentuate the difference in permeabilities between the three regions. The flow is visualized as streamlines flowing left to right colored by the pressure.}
  \label{fig:equilib-obstacle-01-25-flow}
\end{figure}

\subsection{3D experiment}

We also ran a series of 3D experiments. For this experiment, we used the nonconvex domain depicted in Figure~\ref{fig:3d-domain}, which provides a front and back view of the domain. The domain is partitioned into eight unit cubes, each cube forming a subregion with a different permeability. Each subregion is colored according to the inverse permeability (smaller values correspond to higher permeabilities). The permeabilities $K$ span several orders of magnitude, ranging from  $5 \cdot 10^{-5}I$ to $500I$ (inverse permeabilities spanning $2 \cdot 10^4I$ to $2 \cdot 10^{-3}I$). The inverse permeabilities are colored on a log scale in the figure to demonstrate the jumps in orders of magnitude. The fluid has a constant viscosity $\mu = \mu^* = 10^{-3}$ throughout.

\begin{figure}[h!]
  \centering
  \begin{tabular}{cc}
    \includegraphics[width=5cm]{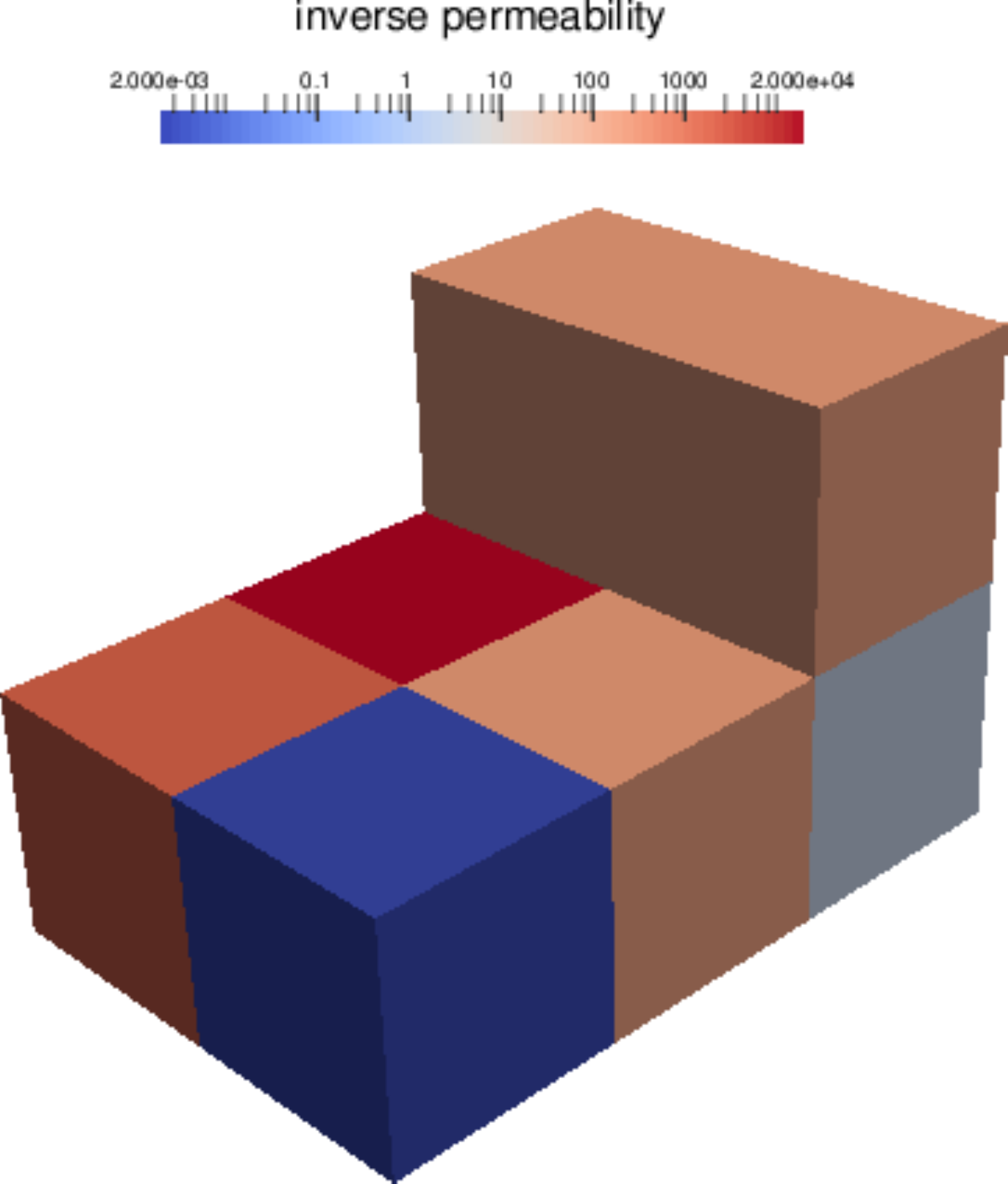} &
    \includegraphics[width=5cm]{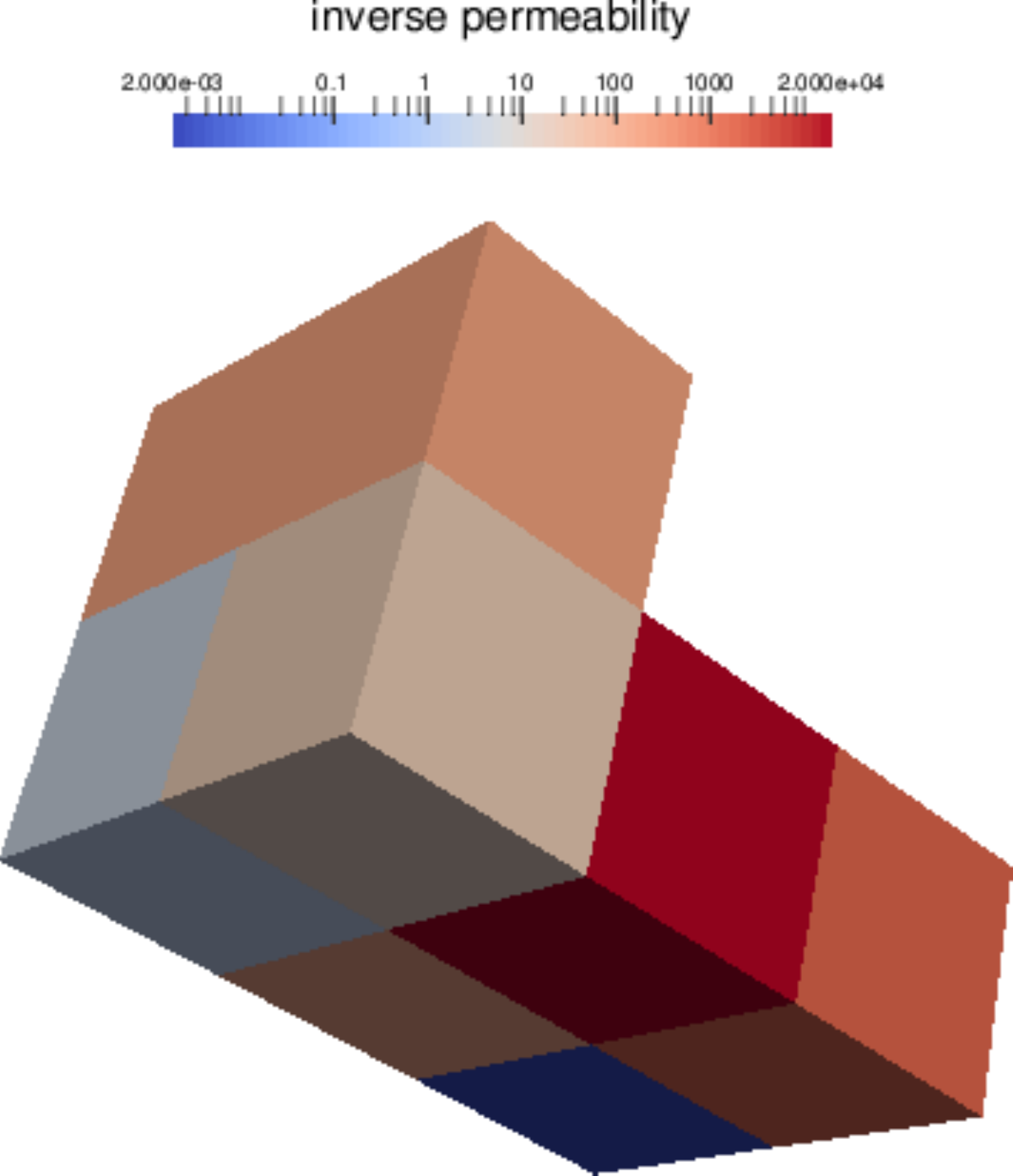}
  \end{tabular}
  \caption{3D nonconvex domain with subregions colored according to inverse permeability (smaller values correspond to higher permeabilities). The permeabilities range from $5 \cdot 10^{-5}I$ to $500I$ (inverse permeabilities ranging from $2 \cdot 10^4I$ to $2 \cdot 10^{-3}I$) and are colored on a log scale to demonstrate the jumps in order of magnitude.}
  \label{fig:3d-domain}
\end{figure}

A 2-dimensional slice of the domain is depicted in Figure~\ref{fig:3d-domain-slice} in order to explain the choice of boundary conditions. All faces of the domain, except for those corresponding to the dotted edges in the figure, have no-flow boundary conditions $\vec{u} = 0$. We define do-nothing boundary conditions on the outflow and a parabolic profile on inflow given independently in the two spatial dimensions with maximal velocity 1 in the center of the face and decreasing quadratically to zero on the walls. The right-hand sides $\vec{f}$ and $g$ are defined to be zero throughout the domain.

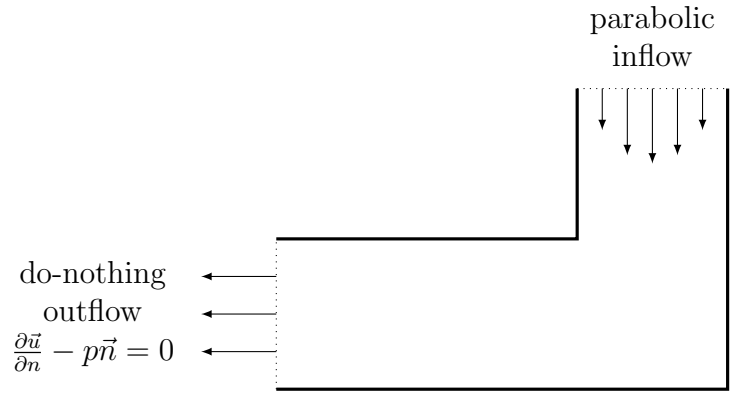
\begin{figure}[h!]
  \centering
  \begin{tikzpicture}[scale=2.0,>=latex]
    \draw[very thick] (0,0) -- (3,0) -- (3,2);
    \draw[very thick] (0,1) -- (2,1) -- (2,2);
    \draw[dotted] (0,0) -- (0,1);
    \draw[dotted] (2,2) -- (3,2);

    \draw[->] (0,0.25) -- (-0.5,0.25);
    \draw[->] (0,0.5) -- (-0.5,0.5);
    \draw[->] (0,0.75) -- (-0.5,0.75);
    \node[anchor=east] at (-0.5,0.5) {\begin{tabular}{c}
        do-nothing \\
        outflow \\
        $\frac{\partial \vec{u}}{\partial n} - p\vec{n} = 0$
    \end{tabular} };

    \draw[->] (2.166667,2) -- (2.166667,1.722222);
    \draw[->] (2.333333,2) -- (2.333333,1.555556);
    \draw[->] (2.500000,2) -- (2.500000,1.500000);
    \draw[->] (2.666667,2) -- (2.666667,1.555556);
    \draw[->] (2.833333,2) -- (2.833333,1.722222);
    \node[anchor=south] at (2.5,2) {\begin{tabular}{c}
        parabolic \\
        inflow
    \end{tabular} };
  \end{tikzpicture}
  \caption{2D slice of the 3D domain, indicating boundary conditions. The parabolic inflow is defined with a maximum velocity magnitude of $1$.}
  \label{fig:3d-domain-slice}
\end{figure}

As with the 2D experiment, we tested the maximum and equilibration strategies with~$(\epsilon,\theta) \in \{0,0.001,0.01\} \times \{0.25,0.5,0.75\}$. There were 2630 degrees of freedom in the initial mesh. For each choice of adaptive strategy, we refined the mesh either a total of 10 times or until the number of degrees of freedom exceeded $10^5$. For comparison, we uniformly refined the mesh 3 times.

Figure~\ref{fig:3d-adapt-vs-uniform} compares the degrees of freedom against the computed error estimates for each experiment in a manner similar to Figure~\ref{fig:2d-adapt-vs-uniform}. The results are similar to the 2D case, showing an improvement in the overall error per degree of freedom for all of the adaptive strategies.

\begin{figure}[h!]
  \centering
  \includegraphics[width=18cm]{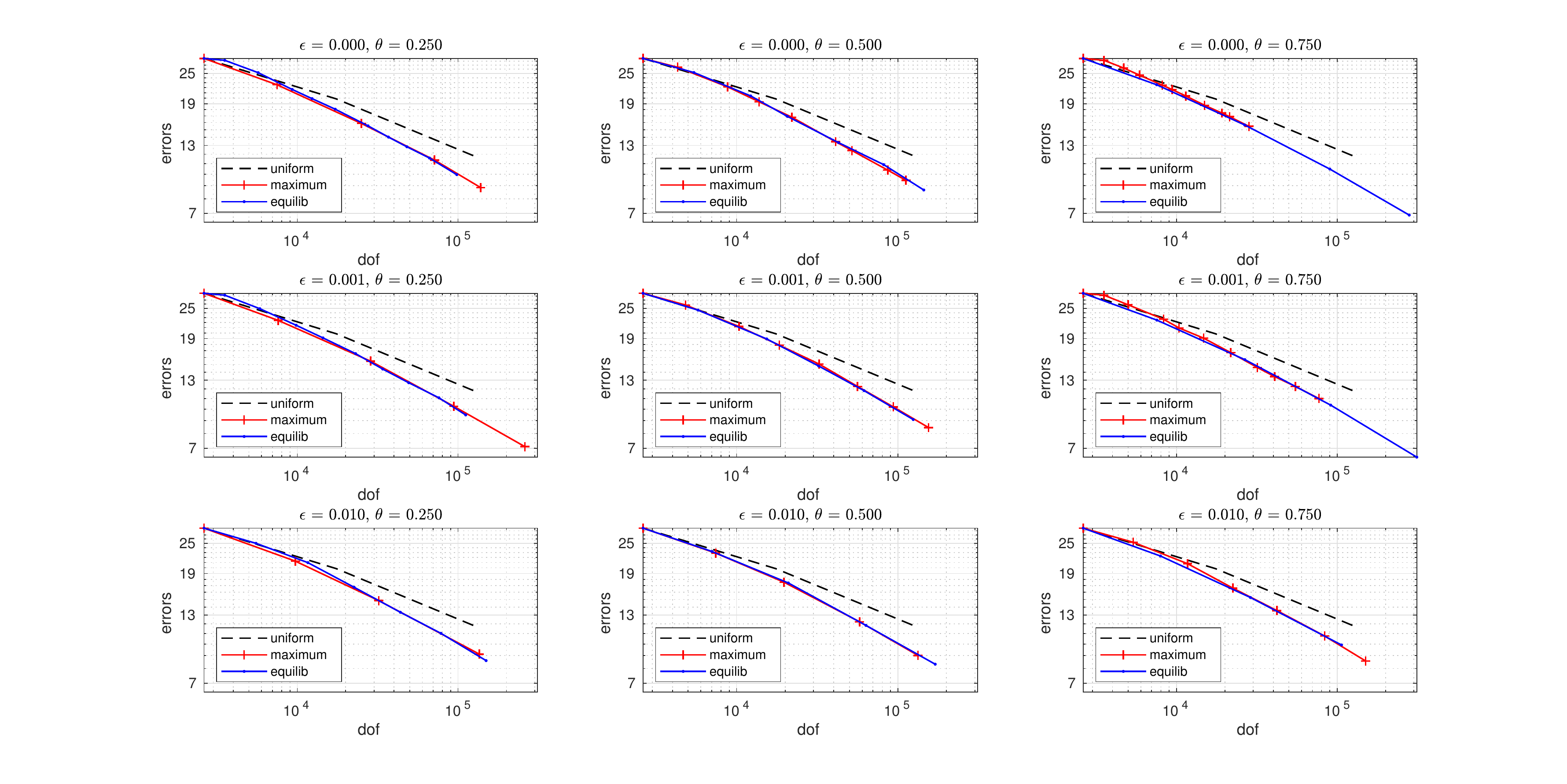}
  \caption{Comparison of each adaptive mesh refinement strategy against uniform refinement for the 3D experiment.}
  \label{fig:3d-adapt-vs-uniform}
\end{figure}

The mesh for the equilibration strategy with $\epsilon = 0.01$ and $\theta = 0.25$ at various levels of refinement is shown in Figure~\ref{fig:equilib-3d-01-25-meshes}. The top left corner shows the initial mesh, the top right after 1 iteration of refinement, the bottom left after 3 iterations, and the bottom right after 6 iterations.

\begin{figure}[h!]
  \centering
  \begin{tabular}{cc}
    \includegraphics[width=6cm]{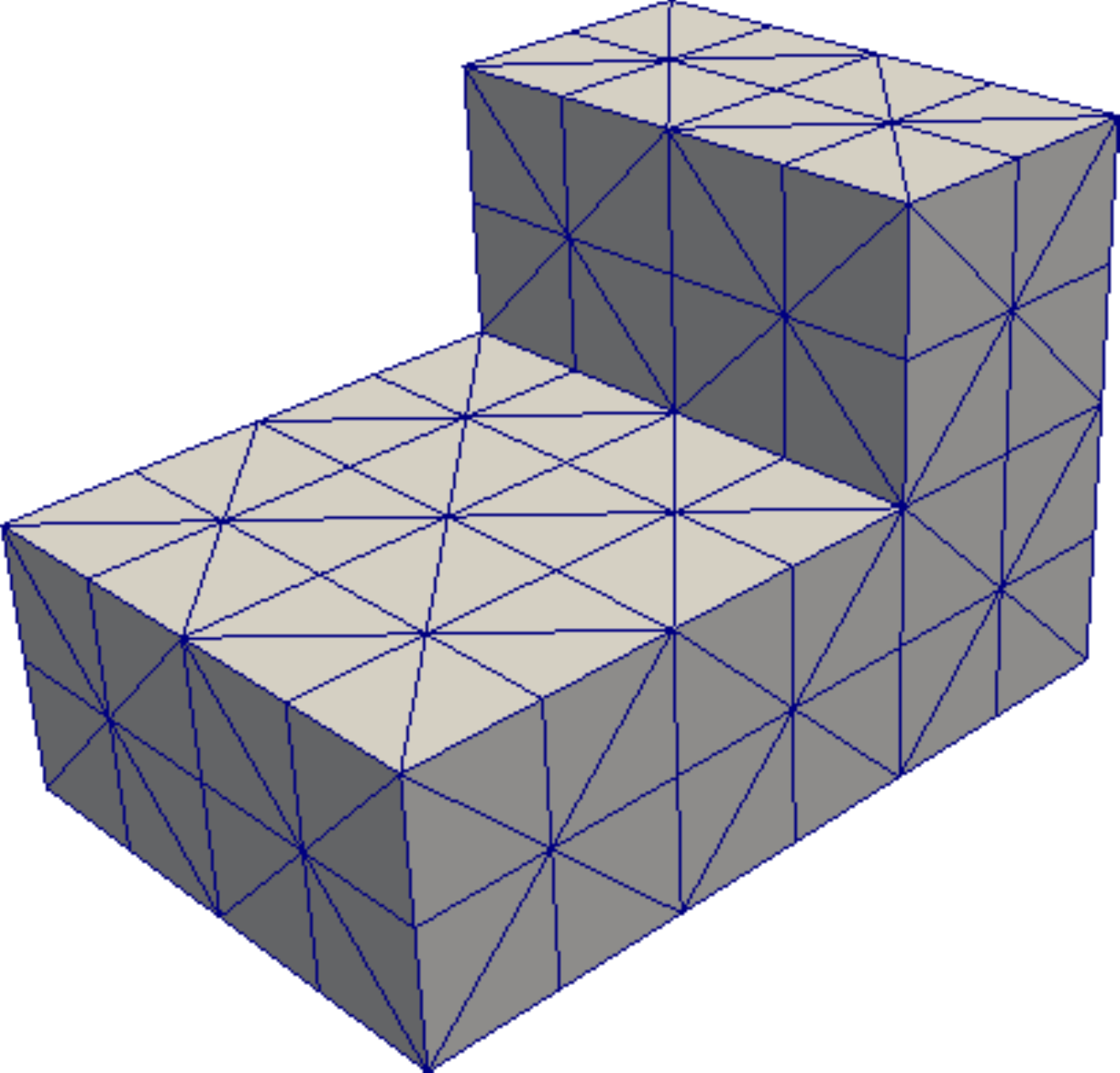} &
    \includegraphics[width=6cm]{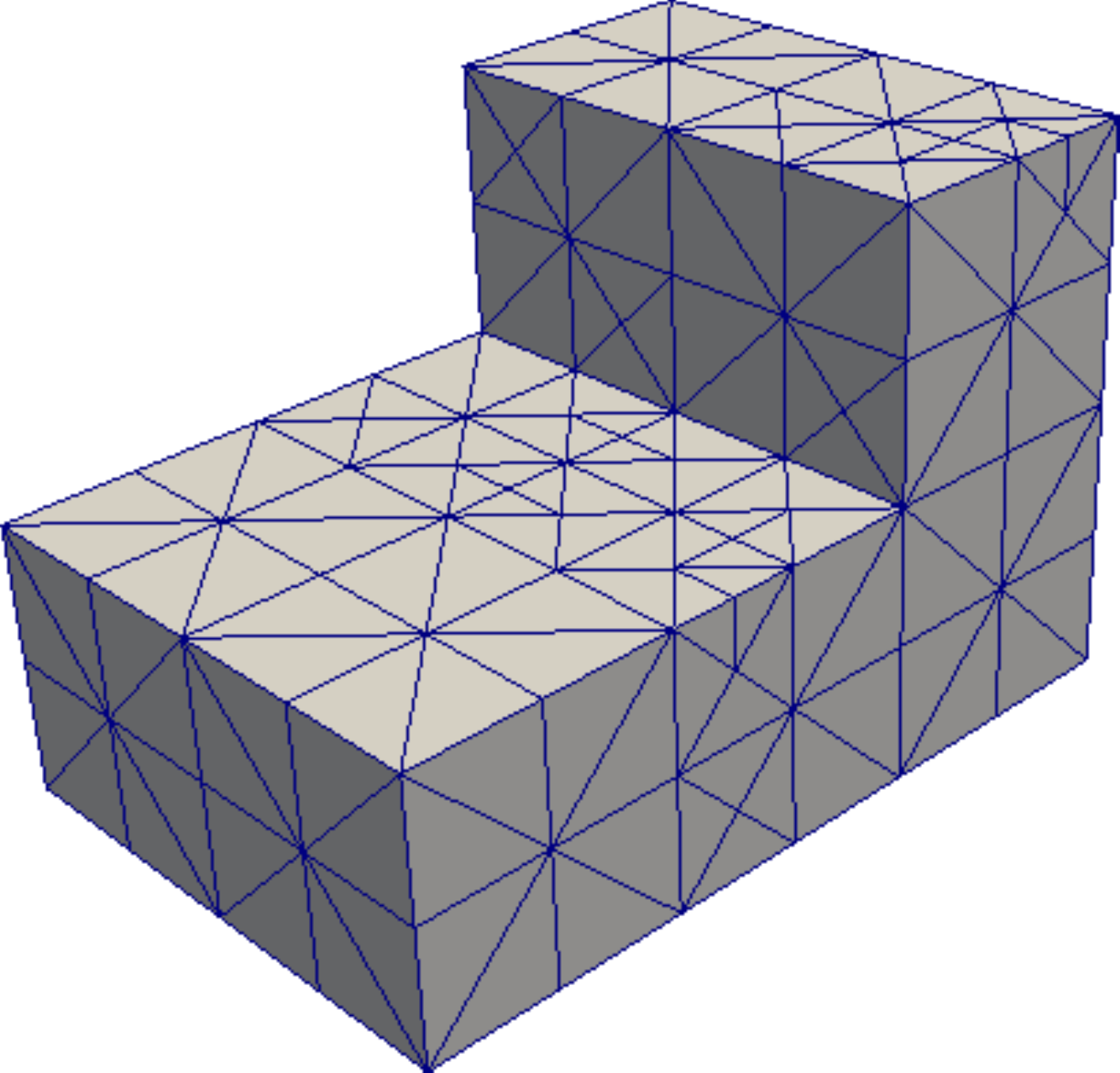} \\
    \includegraphics[width=6cm]{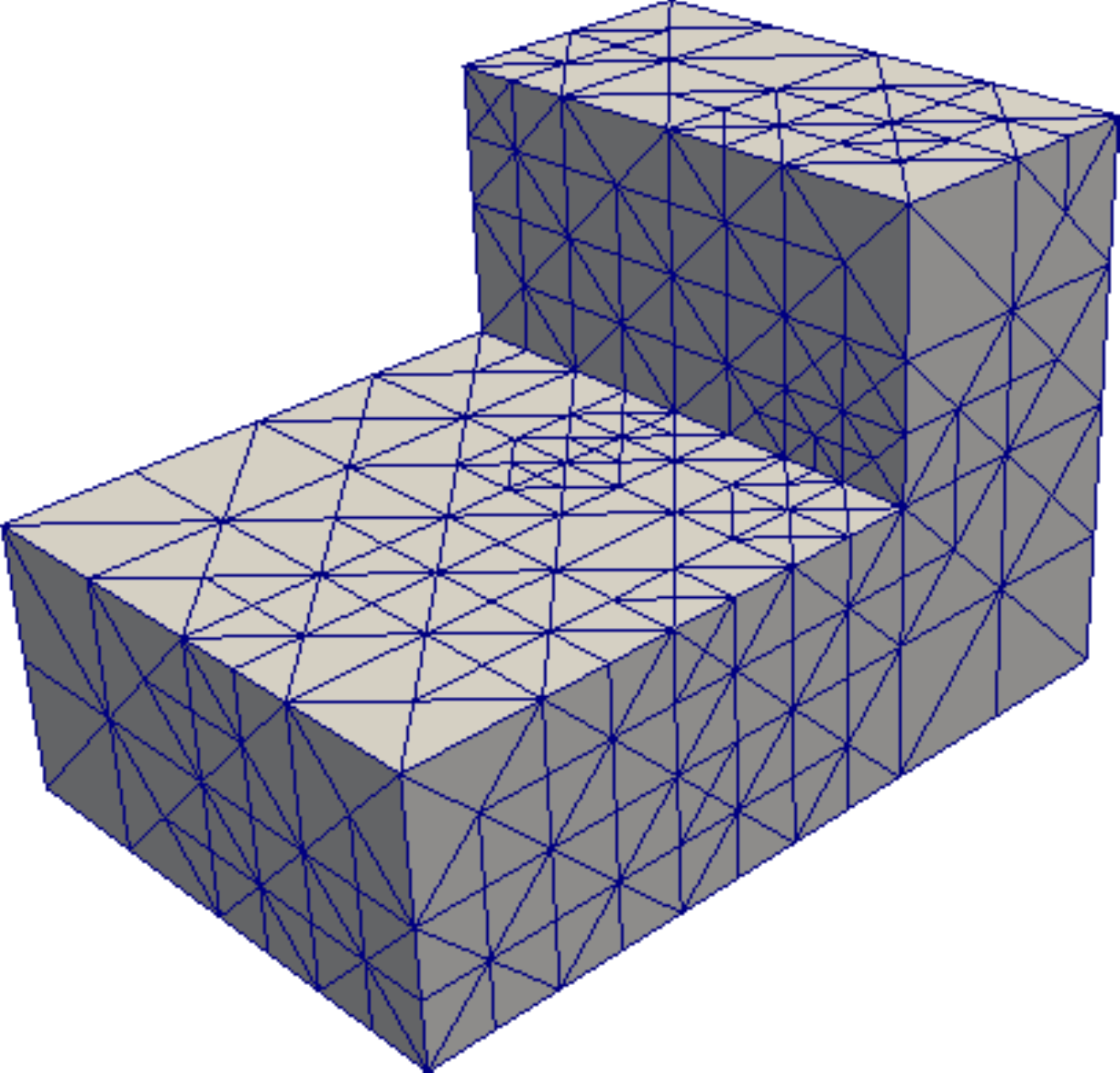} &
    \includegraphics[width=6cm]{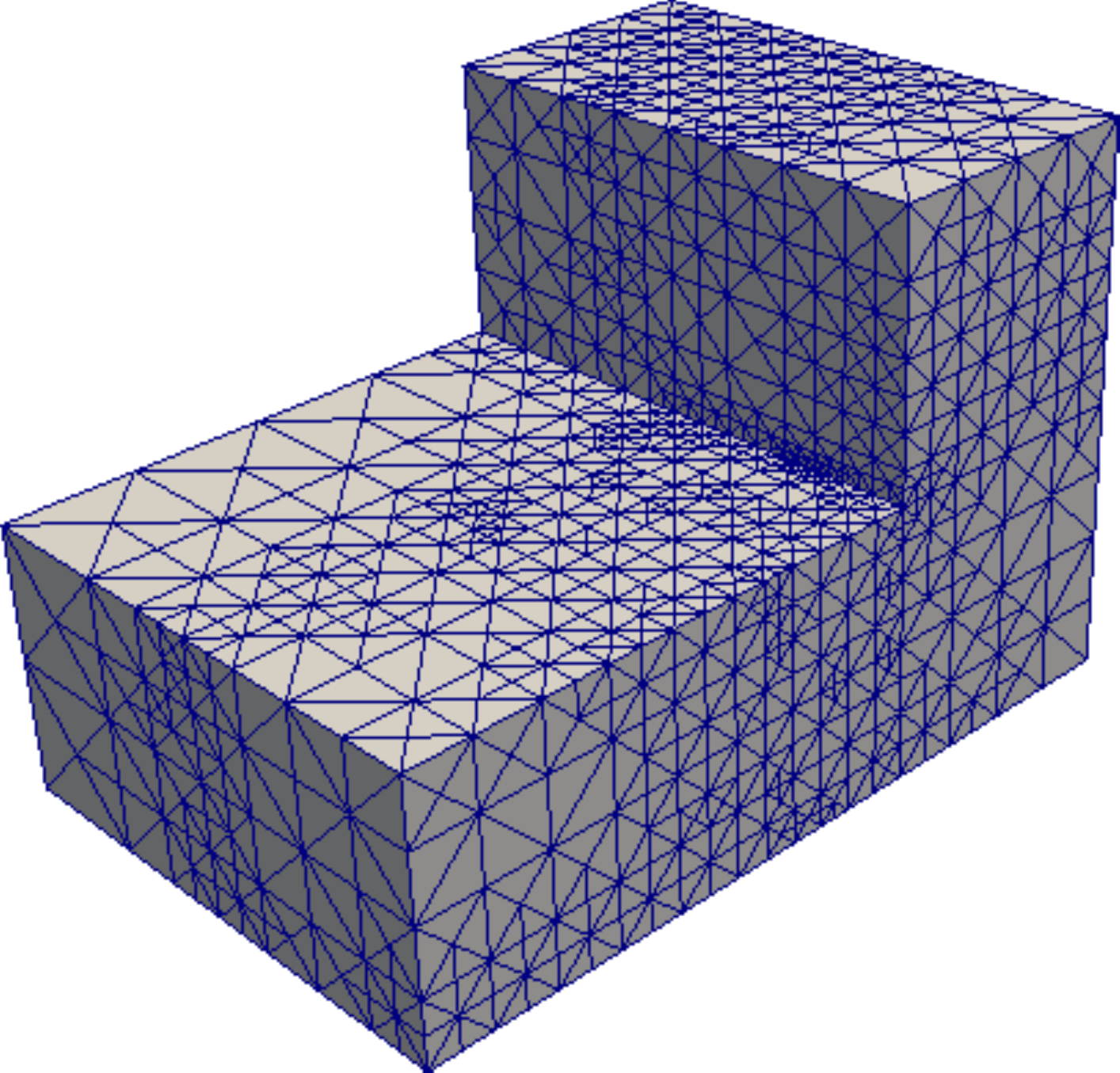}
  \end{tabular}
  \caption{Meshes at various stages of refinement for the 3D equilibration strategy with $\epsilon = 0.01$ and $\theta = 0.25$. The top left corner shows the initial mesh, the top right after 1 iteration of refinement, the bottom left after 3 iterations, and the bottom right after 6 iterations.}
  \label{fig:equilib-3d-01-25-meshes}
\end{figure}

The flow for this experiment is visualized in Figure~\ref{fig:3d-flow}. This visualization used the results from the final iteration of an equilibration strategy with $\epsilon = 0.01$ and $\theta = 0.25$. The domain is colored according to inverse permeability and the flow is visualized as streamlines colored by pressure. 
As expected, the pressure decreases as it travels through the domain. 
Of interest, note that the flow appears to avoid the low permeability (red) subregion of the domain, as it curves to flow into the higher permeability subregion nearby.

\begin{figure}[h!]
  \centering
  \begin{tabular}{cc}
    \includegraphics[width=6.5cm]{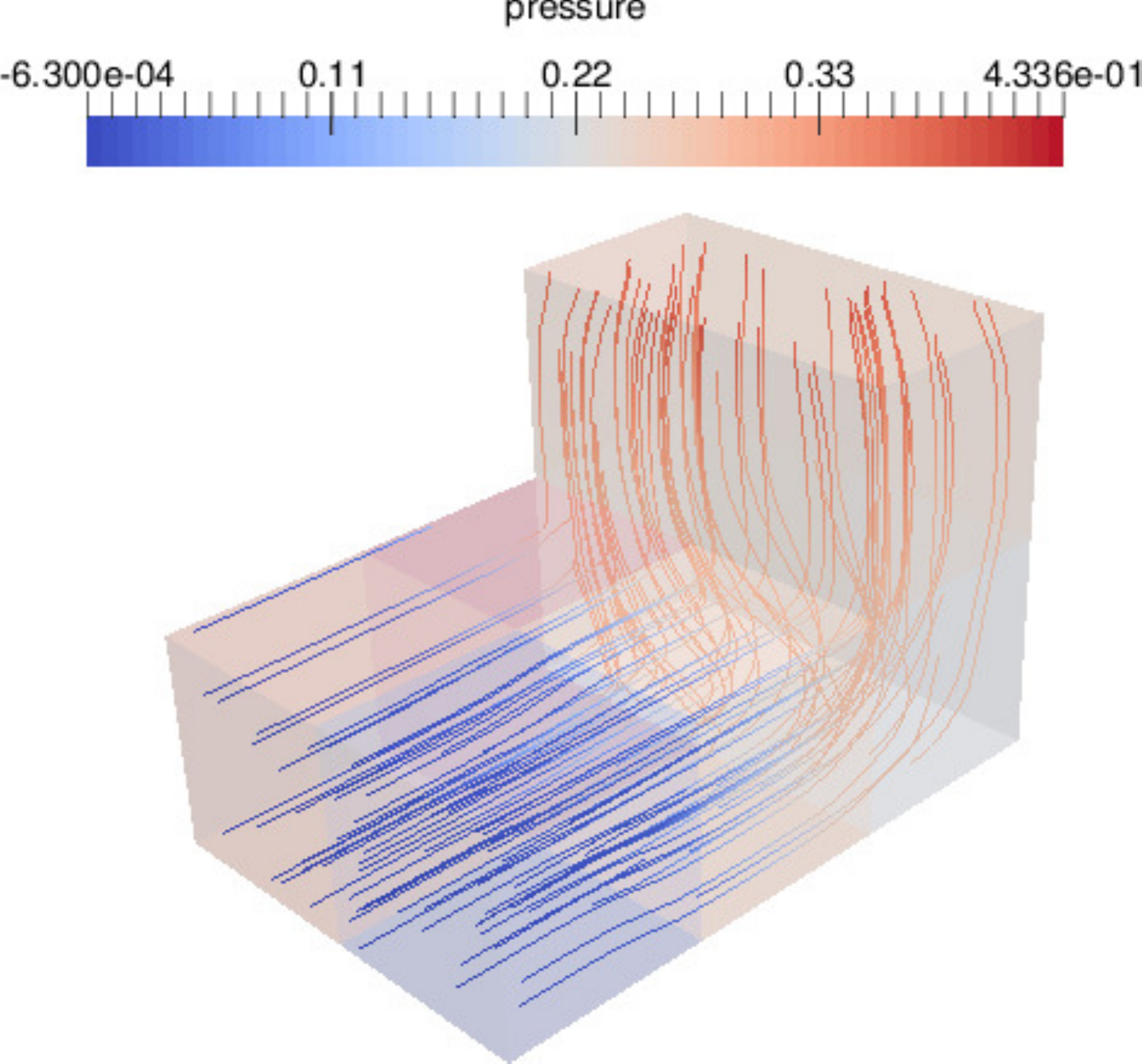} &
    \includegraphics[width=6.1cm]{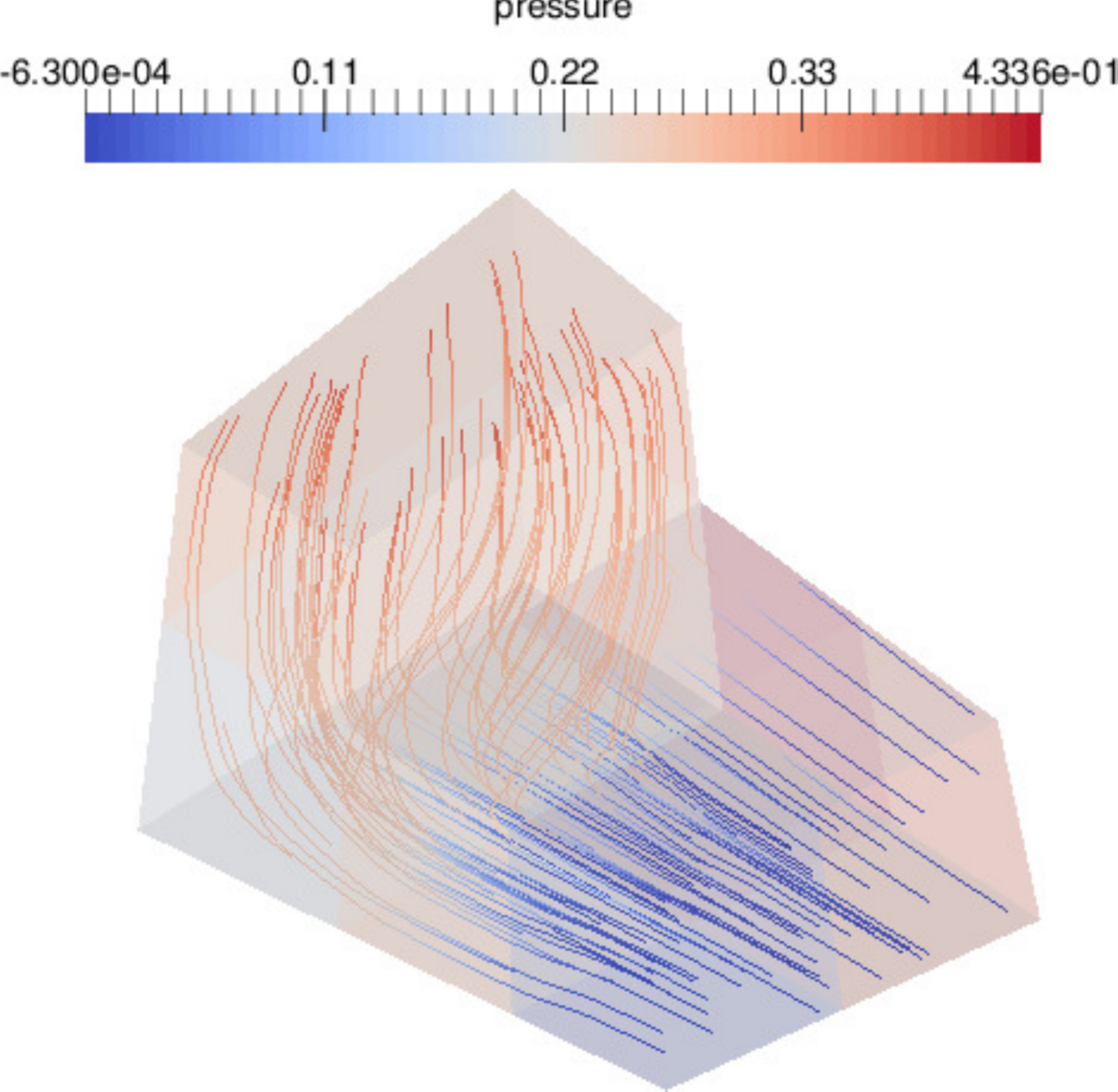}
  \end{tabular}
  \caption{Visualization of flow from final iteration of an equilibration strategy with $\epsilon = 0.01$ and $\theta = 0.25$ for the 3D experiment.}
  \label{fig:3d-flow}
\end{figure}

%% file: aee_paper-conclusion.tex
\section{Conclusion} \label{sec:conclusion}

We extended the error estimate for the Stokes-Brinkman problem developed by Burda and Hasal in~\cite{Burda-2015-AEE} to the 3D domain and to a general right-hand side. We performed numerical experiments in 2D and 3D that showed that the error estimate is effective in driving an adaptive mesh refinement process. The presented error estimate and adaptive mesh refinement strategies will, therefore, be effective in accurately modeling flow using the Stokes-Brinkman equations.

We noticed no substantial difference between the equilibration and maximum refinement strategies, but we noticed that the choice of the parameters $\epsilon$ and $\theta$ have a significant effect on the number of iterations of refinement needed to reach a desired error tolerance. A smaller choice of $\epsilon$ results in fewer elements being refined and therefore more iterations needed to obtain a desired error tolerance. On the other hand, a larger choice of $\epsilon$, while requiring fewer iterations to reach convergence, results in larger problems to be solved. Similar analysis holds for the choice of $\theta$, with a smaller $\theta$ resulting in slower convergence and smaller problems with the equilibration strategy but faster convergence and larger problems with the maximum strategy.

%% file: Stokes-Brinkman-aee.bbl
\begin{thebibliography}{10}

\bibitem{Ainsworth-2000-AEE}
Mark Ainsworth and J.~Tinsley Oden.
\newblock {\em A Posterori Error Estimation in Finite Element Analysis}.
\newblock Wiley, 2000.

\bibitem{Arbogast2007}
Todd Arbogast and Dana~S. Brunson.
\newblock A computational method for approximating a {Darcy-Stokes} system
  governing a vuggy porous medium.
\newblock {\em Computational Geosciences}, 11(3):207--218, Sep 2007.

\bibitem{Paraview}
Utkarsh Ayachit.
\newblock {\em The ParaView Guide: A Parallel Visualization Application}.
\newblock Kitware, 2015.

\bibitem{Bear-1972-PM}
J.~Bear.
\newblock {\em Dynamics of Fluids in Porous Media}.
\newblock Elsevier, 1972.

\bibitem{Beavers-1967-FPM}
G.S. Beavers and D.D. Joseph.
\newblock Boundary conditions at a naturally permeable wall.
\newblock {\em J. Fluid Mech}, 30:197–--207, 1967.

\bibitem{Brezzi-1991-MHF}
F.~Brezzi and M.~Fortin.
\newblock {\em Mixed and Hybrid Finite Element Methods}.
\newblock Springer-Verlag, New York -- Berlin -- Heidelberg, 1991.

\bibitem{Brinkman-1948-SBE}
H.~C. Brinkman.
\newblock A calculation of the viscous force exerted by a flowing fluid on a
  dense swarm of particles.
\newblock {\em Appl. Sci. Res.}, A1:27--34, 1948.

\bibitem{Burda-2000-AEE}
Pavel Burda.
\newblock An a posteriori error estimate for the {Stokes} problem in a
  polygonal domain using {Hood-Taylor} elements.
\newblock In P.~Neittaanm\"{a}ki, T.~Tiihonen, and P.~Tarvainen, editors, {\em
  Proceedings of the 3rd European Conference on Numerical Mathematics and
  Advanced Applications, ENUMATH 99, Jyv\"{a}skyl\"{a}, Finland, 26--30 July
  1999}, pages 448--455, Singapore, 2000. World Scientific.

\bibitem{Burda-2001-AEE}
Pavel Burda.
\newblock A posteriori error estimates for the {Stokes} flow in {2D} and {3D}
  domains.
\newblock In P.~Neittaanm\"{a}ki and M.~K\v{r}\'{i}\v{z}ek, editors, {\em
  Finite Element Methods, 3D Problems}, {GAKUTO} Int. Ser., Math. Sci.
  Appl.~15, pages 34--44, 2001.

\bibitem{Burda-2015-AEE}
Pavel Burda and Martin Hasal.
\newblock An a posteriori error estimate for the {Stokes-Brinkman} problem in a
  polygonal domain.
\newblock In {\em Programs and Algorithms of Numerical Mathematics}, pages
  32--40. Institute of Mathematics AS CR, 2015.

\bibitem{Burda-2003-AEE}
Pavel Burda, Jaroslav Novotn\'{y}, and Bed\v{r}ich Soused\'{i}k.
\newblock A posteriori error estimates applied to flow in a channel with
  corners.
\newblock {\em Math. Comput. Simulat.}, 61:375--383, 2003.

\bibitem{Ciarlet-2002-FEM}
Philippe~G. Ciarlet.
\newblock {\em The Finite Element Method for Elliptic Problems}, volume~40 of
  {\em Classics in Applied Mathematics}.
\newblock Society for Industrial and Applied Mathematics (SIAM), Philadelphia,
  PA, 2002.
\newblock Reprint of the 1978 original [North-Holland, Amsterdam].

\bibitem{Darcy-1856-FPM}
H.~Darcy.
\newblock {\em Les fontaines publiques de la ville de {Dijon}}.
\newblock Dalmont, Paris, 1856.

\bibitem{Elman-2014-FEF}
H.~C. Elman, D.~J. Silvester, and A.~J. Wathen.
\newblock {\em Finite elements and fast iterative solvers: with applications in
  incompressible fluid dynamics}.
\newblock Oxford University Press, New York, second edition, 2014.

\bibitem{Eriksson-1995-IAM}
Kenneth Eriksson, Don Estep, Peter Hansbo, and Claes Johnson.
\newblock Introduction to adaptive methods for differential equations.
\newblock {\em Acta Numerica}, 4:105--158, 1995.

\bibitem{Gulbransen-2010-MMF}
Astrid~Fossum Gulbransen, Vera~Louise Hauge, and Knut-Andreas Lie.
\newblock A multiscale mixed finite-element method for vuggy and
  naturally-fractured reservoirs.
\newblock {\em SPE Journal}, 15(2):395--403, 2010.

\bibitem{Laptev-2003-PPM}
Vsevolod Laptev.
\newblock {\em Numerical solution of coupled flow in plain and porous media}.
\newblock PhD thesis, Technical University of Kasierslautern, 2003.

\bibitem{Popov-2007-MMM}
P.~Popov, L.~Bi, Y.~Efendiev, R.~E. Ewing, G.~Qin, J.~Li, and Y.~Ren.
\newblock Multi-physics and multi-scale methods for modeling fluid flow through
  naturally-fractured vuggy carbonate reservoirs.
\newblock In {\em Proceedings of the 15th SPE Middle East Oil \& Gas Show and
  Conference, Bahrain, 11-14 March}, 2007.
\newblock (Paper SPE 105378).

\bibitem{Schoberl-1997}
Joachim Sch{\"o}berl.
\newblock {NETGEN} - an advancing front {2D/3D}-mesh generator based on
  abstract rules.
\newblock {\em Computing and Visualization in Science}, 1(1):41--52, Jul 1997.

\bibitem{URQUIZA2008525}
J.M. Urquiza, D.~N'Dri, A.~Garon, and M.C. Delfour.
\newblock Coupling {Stokes} and {Darcy} equations.
\newblock {\em Applied Numerical Mathematics}, 58(5):525 -- 538, 2008.

\bibitem{Verfurth-2013-AEE}
R.~Verf{\"u}rth.
\newblock {\em A Posteriori Error Estimation Techniques for Finite Element
  Methods}.
\newblock A Posteriori Error Estimation Techniques for Finite Element Methods.
  OUP Oxford, 2013.

\bibitem{Whitaker-1986-FPM}
Stephen Whitaker.
\newblock Flow in porous media {I}: A theoretical derivation of {Darcy's} law.
\newblock {\em Transport in Porous Media}, 1(1):3--25, Mar 1986.

\end{thebibliography}
